\documentclass[12pt]{amsart}

\usepackage{amsmath,xypic,xcolor}
\usepackage{amssymb}
\usepackage{hyperref}
\usepackage[margin=1in]{geometry}
\usepackage{setspace}

\theoremstyle{plain}
  \newtheorem{lemma}[equation]{Lemma}
  \newtheorem{proposition}[equation]{Proposition}
  \newtheorem{theorem}[equation]{Theorem}
  \newtheorem{corollary}[equation]{Corollary} 
  \newtheorem{conjecture}[equation]{Conjecture}

\theoremstyle{definition}
  \newtheorem{definition}[equation]{Definition}

\theoremstyle{remark}
  \newtheorem{remark}[equation]{Remark}
  \newtheorem{example}[equation]{Example}

\makeatletter
\@addtoreset{equation}{section}
\makeatother

\DeclareMathOperator{\Aut}{Aut}
\DeclareMathOperator{\Alb}{Alb}
\DeclareMathOperator{\Amp}{Amp}
\DeclareMathOperator{\Bs}{Bs}
\DeclareMathOperator{\Hilb}{Hilb} 
\DeclareMathOperator{\Nef}{Nef}
\DeclareMathOperator{\NEb}{\overline{NE}}
\DeclareMathOperator{\Pic}{Pic}
\DeclareMathOperator{\rank}{rank}
\DeclareMathOperator{\Vol}{Vol}

\newcommand{\BB}{\mathbf B}
\newcommand{\CC}{\mathbb C}
\newcommand{\E}{\mathcal E}
\newcommand{\F}{\mathcal F}
\renewcommand{\L}{\mathcal L}
\renewcommand{\O}{\mathcal O}
\newcommand{\PP}{\mathbb P}
\newcommand{\QQ}{\mathbb Q}
\newcommand{\QQbar}{\overline{\mathbb Q}}
\newcommand{\RR}{\mathbb R}
\newcommand{\ZZ}{\mathbb{Z}}

\newcommand{\abs}[1]{\left\lvert #1 \right\rvert}
\newcommand{\dabs}[1]{\left\lVert #1 \right\rVert}
\newcommand{\set}[1]{\left\{ #1 \right\}}

\renewcommand{\emptyset}{\varnothing}
\newcommand{\id}{\mathrm{id}}
\newcommand{\rat}{\dashrightarrow}
\renewcommand{\setminus}{\smallsetminus}

\newif\ifhascomments \hascommentstrue
\ifhascomments
\newcommand{\matt}[1]{{\color{red}[[\ensuremath{\spadesuit\spadesuit\spadesuit} #1]]}}
\else
  \newcommand{\matt}[1]{}
\fi

\title[Hyper-K\"ahler canonical heights and Kawaguchi--Silverman]{Canonical heights on hyper-K\"ahler varieties and the Kawaguchi-Silverman conjecture}

\author{John Lesieutre}
\address{John Lesieutre\\ University of Illinois at Chicago \\ Math, Stat \& CS, Room 411 SEO \\ 851 S Morgan St, M/C 249 \\Chicago, IL 60607-7045 \\ USA}
\email{jdl@uic.edu}
\thanks{J.L. is partially supported by NSF grants DMS-1700898 and DMS-1246844.}
\author{Matthew Satriano}
\address{Matthew Satriano\\ Department of Pure Mathematics\\ University of Waterloo\\ Waterloo, ON N2L 3G1\\ Canada}
\email{msatrian@uwaterloo.ca}
\thanks{M.S. is partially supported by NSERC grant RGPIN-2015-05631.}

\begin{document}

\begin{abstract}
The Kawaguchi--Silverman conjecture predicts that if \(f : X \rat X\) is a dominant rational-self map of a projective variety over \(\QQbar\), and \(P\) is a \(\QQbar\)-point of \(X\) with Zariski-dense orbit, then the dynamical and arithmetic degrees of \(f\) coincide: \(\lambda_1(f) = \alpha_f(P)\).  We prove this conjecture in several higher-dimensional settings, including all endomorphisms of non-uniruled smooth projective threefolds with degree larger than \(1\), and all endomorphisms of hyper-K\"ahler varieties in any dimension.  In the latter case, we construct a canonical height function associated to any automorphism \(f : X \to X\) of a hyper-K\"ahler variety defined over \(\QQbar\).
\end{abstract}

\maketitle
\tableofcontents

\section{Introduction}

Let \(f : X \rat X\) be a dominant rational self-map of a smooth projective variety \(X\) defined over $\QQbar$.  There are two natural degree functions one can associate to the dynamical system $(X,f)$. The first measures the growth rate of the degrees of the iterates \(f^{n}\). It is known as the \emph{first dynamical degree} and is defined as
\[
\lambda_1(f) = \lim_{n \to \infty} \left( (f^n)^\ast H \cdot H^{\dim X-1} \right)^{1/n}
\]
where $H$ is a choice of ample divisor on $X$; a result of Dinh and Sibony~\cite{dinhsibony} says that this limit exists and is independent of the choice of ample divisor \(H\). The second notion is the \emph{arithmetic degree}, which depends on a choice of $\QQbar$-point $P$, and reflects the growth rate of the heights the points \(f^n(P)\). Letting $h_H$ denote the logarithmic Weil height associated to $H$, and $h_H^+=\max(h_H,1)$, we define
\[
\underline{\alpha}_f(P) = \liminf_{n \to \infty} h^+_H(f^n(P))^{1/n}, \quad\quad
\overline{\alpha}_f(P) = \limsup_{n \to \infty} h^+_H(f^n(P))^{1/n}.
\]

Both of these quantities are again independent of the choice of ample divisor \(H\) \cite[Proposition 12]{dynamical-arithmetic-rat-maps} and it is conjectured that these two quantities always coincide. When they do, \(\alpha_f(P)\) is defined to be the common value. Whether or not $\underline{\alpha}_f(P)$ and $\overline{\alpha}_f(P)$ are equal remains open in general, but it is known when \(f\) is a morphism~\cite[Theorem 3]{ks-jordan}, which will always be the case in this paper. The Kawaguchi-Silverman conjecture is then:

\begin{conjecture}[Kawaguchi--Silverman \cite{dynamical-arithmetic-rat-maps}]
\label{conj:main} 
Let \(X\) be a smooth projective variety and let \(f : X \rat X\) be a dominant rational map defined over \(\QQbar\). Suppose that \(P\) is a \(\QQbar\)-point of \(X\).  If the forward orbit of \(P\) under \(f\) is Zariski dense, then \(\alpha_f(P)\) exists and is equal to \(\lambda_1(f)\).
\end{conjecture}

The conjecture is known in many cases; see \cite[Remark 1.8]{mss} for a comprehensive list including abelian varieties \cite{ks-jordan,silverman-ab-vars}, automorphisms of smooth projective surfaces \cite{surface-aut,dynamical-equals-arithmetic}, as well as certain product varieties \cite{sano-prod-endo}. Recently the conjecture was proved for all regular endomorphisms of smooth projective surfaces~\cite{mss}; the proof in the case of surfaces relies heavily on the birational classification.

Our aim in this paper is to prove the conjecture in several higher dimensional settings. A basic difficulty is that the classification of $n$-folds (for $n\geq3$) is much more difficult and less complete than the classification of surfaces: there is no neat analog of the Enriques--Kodaira classification, and one must instead attempt to understand the interplay between the geometry of endomorphisms and the classification theory of higher-dimensional varieties. 

We break up our analysis according to Kodaira dimension $\kappa(X)$. We note that by \cite[Theorem A]{nakayamazhang}, if $\kappa(X)>0$, then an iterate of $f$ preserves the Iitaka fibration and so there is no \(\QQbar\)-point \(P\) on \(X\) with a Zariski dense orbit; as a result Conjecture \ref{conj:main} vacuously holds. Thus, the only remaining cases to consider are those of Kodaira dimension $0$ and $-\infty$.

Let us now discuss in detail our main results as well as several consequences. We say a smooth projective variety $X$ is \emph{Calabi--Yau} if \(\dim X \geq 3\), \(\O_X(K_X)\) is trivial and \(h^0(\Omega^p_X) = 0\) for \(0 < p < n\). We say $X$ is \emph{hyper-K\"ahler} if its complex analytification is simply connected and \(H^0(X,\Omega^2_X)\) is spanned by a symplectic form.

\begin{theorem}
\label{thm:hk}
Conjecture \ref{conj:main} is true for surjective endomorphisms of hyper-K\"ahler varieties.
\end{theorem}

The key to proving Theorem \ref{thm:hk} is to construct a canonical height function associated to $f$, following a strategy developed by Silverman~\cite{silvermank3} and Kawaguchi~\cite{surface-aut} in dimension \(2\). Along the way, we obtain a hyper-K\"ahler analog of a result of Cantat and Kawaguchi~\cite[Proposition 4.1]{cantatsurvey}, the latter having been shown for surfaces. In the statement below, if $V$ is $f$-periodic with \(f^n(V) = V\), then \(\lambda_1(f\vert_V)\) is interpreted as \(\lambda_1(\widetilde{f^n} \vert_{\widetilde{V}})^{1/n}\), where \(\widetilde{f^n} : \widetilde{V} \to \widetilde{V}\) is the induced automorphism of the normalization \(\widetilde{V}\) of \(V\).

\begin{theorem}
\label{intro:contract}
Suppose that \(f: X \to X\) is an automorphism of a hyper-K\"ahler variety satisfying \(\lambda_1(f) > 1\).  Let
\[
E(f) = \bigcup \set{ V : \text{$\dim V \geq 1$, $V$ is $f$-periodic, $\lambda_1(f \vert_V) < \lambda_1(f)$, and $\lambda_1(f^{-1} \vert_V) < \lambda_1(f^{-1})$}} 
\]
Then \(E(f)\) is not Zariski-dense in \(X\), and there exists a morphism \(\pi : X \to Y\) which contracts every connected component of \(E(f)\) to a point.
\end{theorem}

As a consequence of Theorem \ref{thm:hk}, combined with the work of Sano \cite{sano-prod-endo}, we are able to show that the conjecture holds for automorphisms of all varieties with \(K_X \equiv 0\) as long as it holds for automorphisms of Calabi--Yau varieties.

\begin{corollary}
\label{cor:min Kod 0 <-> CYn}
Let $n$ be a positive integer. Then Conjecture \ref{conj:main} is true for all automorphisms of smooth projective varieties $X$ with dimension at most $n$ and $K_X$ numerically trivial if and only if Conjecture \ref{conj:main} is true for all automorphisms of smooth Calabi--Yau varieties with dimension at most $n$.
\end{corollary}

\begin{remark}
\label{rmk:min Kod 0 <-> CYn}
The Abundance Conjecture implies that every smooth projective minimal variety $X$ of Kodaira dimension 0 has $K_X$ numerically trivial. Therefore, assuming the Abundance Conjecture in dimension at most $n$, Corollary \ref{cor:min Kod 0 <-> CYn} reduces Conjecture \ref{conj:main} for automorphisms of smooth projective minimal varieties of Kodaira dimension 0 to the special case of smooth Calabi--Yau varieties.
\end{remark}

In the case of dimension 3, we obtain more detailed results for \emph{endomorphisms} as well as automorphisms. Using results of Fujimoto \cite{fujimoto}, we show:

\begin{proposition}
\label{prop:intro-3-fold-Kod-dim-0}
Conjecture \ref{conj:main} holds for all surjective endomorphisms $f\colon X\to X$ of degree $\deg(f)>1$ on smooth projective threefolds $X$ of Kodaira dimension 0 .
\end{proposition}

Since the Abundance Conjecture is known in dimension 3 ~\cite{kawamata-abundance}, by Corollary \ref{cor:min Kod 0 <-> CYn} and Remark \ref{rmk:min Kod 0 <-> CYn}, to prove the conjecture for automorphisms of smooth minimal 3-folds of Kodaira dimension 0, it is enough to handle the case of automorphisms of smooth Calabi--Yau 3-folds. As such, we turn to the case of Calabi--Yau 3-folds and prove the following technical result:

\begin{theorem}
\label{thm:intro-CY3-aut}
Let $f$ be an automorphism of a smooth Calabi--Yau threefold $X$. Suppose that either
\begin{enumerate}
\item\label{intro-CY3-aut::Miyaoka} $c_2(X)$ is positive on $\Nef(X)$, or
\item\label{intro-CY3-aut::semiample} there is a non-zero semi-ample class $D\in\Nef(X)\cap N^1(X)$ such that $c_2(X)\cdot D=0$.
\end{enumerate}
Then Conjecture \ref{conj:main} holds for $(X,f)$.
\end{theorem}

\begin{remark}[Understanding the hypotheses of Theorem \ref{thm:intro-CY3-aut}]
\label{rmk:intro-CY3-aut}
The essential point here is that Theorem \ref{thm:intro-CY3-aut} applies to all Calabi--Yau 3-folds with sufficiently large Picard number, assuming \cite[Question-Conjecture 2.6]{semi-ampleness-conj} and the semi-ampleness conjecture \cite[Conjecture 2.1]{CY-semiample}. In particular, by Remark \ref{rmk:min Kod 0 <-> CYn}, this would resolve Conjecture \ref{conj:main} for all automorphisms of smooth minimal 3-folds of Kodaira dimension 0 and sufficiently large Picard number.

Let us explain why this is the case. A theorem of Miyaoka \cite{Miyaoka} shows that either hypothesis (\ref{intro-CY3-aut::Miyaoka}) of Theorem \ref{thm:intro-CY3-aut} holds or $F:=c_2(X)^\perp\cap\Nef(X)$ is a non-zero face of the nef cone of $X$. In \cite[Question-Conjecture 2.6]{semi-ampleness-conj}, Oguiso asks if $F$ must always be rational when the Picard number $\rho(X)$ is sufficiently large. Provided this is true, there would be a non-zero rational class $D\in F$, and then the semi-ampleness conjecture \cite[Conjecture 2.1]{CY-semiample} would tell us that after scaling $D$ by a positive integer, we can assume it is semi-ample, i.e.~hypothesis (\ref{intro-CY3-aut::semiample}) holds.
\end{remark}

Finally, we turn to the case of Kodaira dimension $-\infty$. Here the closest analogue of a minimal variety is one which has the structure of a Mori fiber space; this includes, for example, all rational normal scrolls. We prove Conjecture \ref{conj:main} in two special cases:

\begin{theorem}
\label{thm:intro-MFS-partial}
Conjecture \ref{conj:main} holds for the following cases:

\begin{enumerate}
\item\label{intro-MFS-partial::auts3fold} all automorphisms of 3-folds which have the structure of a Mori fiber space.

\item\label{intro-MFS-partial::ratnormalscrolls} all surjective endomorphisms of $n$-fold rational normal scrolls.
\end{enumerate}
\end{theorem}

\begin{remark}
\label{rmk:PE-->semistable case}
In the process of showing Theorem \ref{thm:intro-MFS-partial} (\ref{intro-MFS-partial::ratnormalscrolls}), we in fact prove a stronger result: if $C$ is a smooth curve, then Conjecture \ref{conj:main} holds for all surjective endomorphisms of all projective bundles $\PP_C(\E)$ if and only if it holds in the case where $\E$ is semistable of degree 0. See Corollary \ref{cor:reduction to case of semistable deg 0}.
\end{remark}

It is worth mentioning that in Section \ref{preliminaries}, we prove general results concerning the following set-up: $\pi\colon X\to Y$ is a surjective morphism of normal projective varieties over $\QQbar$, $f$ is a surjective endomorphism of $X$, $g$ is a surjective endomorphism of $Y$, and $\pi\circ f=g\circ\pi$. We give several criteria by which one can reduce the conjecture for $(X,f)$ to that of $(Y,g)$, see Theorem \ref{thm:fibered-over-dim-one-less}. 

\section*{Acknowledgments}
We are grateful to Lawrence Ein and Mihai Fulger for useful remarks.

\section{Interplay between Conjecture \ref{conj:main}, fibrations, and birational maps}
\label{preliminaries}
We first recall the main defintions.   We work throughout over \(\QQbar\),
and where not otherwise stipulated, a variety is assumed to be defined over \(\QQbar\).
Given a projective variety \(X\), we write \(N^1(X)\) for the N\'eron--Severi group and \(N^1(X)_\RR = N^1(X) \otimes \RR\) for the corresponding finite-dimensional \(\RR\)-vector space.  We use \(\sim\) for the relation of linear equivalence of Cartier divisors, \(\sim_{\RR}\) for \(\RR\)-linear equivalence, and \(\equiv\) for numerical equivalence. Rational maps are denoted by ``$\rat$'' and morphisms by ``$\to$''.

Suppose that \(f : X \rat X\) is a dominant rational map of a smooth projective variety, and fix an ample divisor \(H\) on \(X\).
As discussed in the introduction, the first dynamical degree is \[\lambda_1(f) = \lim_{n \to \infty} \left( (f^n)^\ast H \cdot H^{\dim X-1} \right)^{1/n}.\] In general, the limit is difficult to compute, since \((f^n)^\ast\) does not necessarily coincide with \((f^\ast)^n\) for rational maps.  However, if \(f : X \to X\) is a morphism, then \((f^n)^\ast = (f^\ast)^n\) and
\[
\lambda_1(f) = \operatorname{SpecRad}(f^\ast : N^1(X)_\RR \to N^1(X)_\RR)
\]
is simply the spectral radius of \(f^\ast\), i.e.~the absolute value of the largest eigenvalue of $f^\ast$ acting on $N^1(X)_\RR$.  When \(f\) is a morphism, we may also drop the smoothness hypothesis on \(X\), and it suffices to assume that \(X\) is normal: there is no difficulty in pulling back Cartier divisors.

Invariant fibrations play an important role in the study of rational maps in higher dimension, and the product formula of Dinh, Nguy\^en, and Truong~\cite{dinhnguyentruong} is useful in dealing with their dynamical degrees.  Suppose that there exists a surjective morphism \(\pi : X \to Y\) and a dominant rational map \(g : Y \rat Y\) with \(g \circ \pi = \pi \circ f\).  Let \(H'\) be an ample divisor on \(Y\).
\[
\xymatrix{
X \ar[d]^\pi \ar@{-->}[r]^-{f} & X \ar[d]^\pi \\
Y \ar@{-->}[r]_-{g} & Y
}
\]

\begin{definition}
The \emph{first dynamical degree of} \(f\) \emph{relative to} \(\pi\) is the limit
\[
\lambda_1(\pi \vert_f) = \lim_{n \to \infty} \left( (f^n)^\ast H \cdot \pi^\ast (H'^{\dim Y}) \cdot H^{\dim X - \dim Y - 1} \right)^{1/n}.
\]
\end{definition}

The definition of relative dynamical degrees can be extended even to the setting in which \(\pi\) itself is only a dominant rational map, although we will not require it. 
The basic properties of dynamical degrees and their relative counterparts are worked out in \cite{dinhsibony, dinhsibonygreen, dinhnguyen, dinhnguyentruong}; a more algebro-geometric perspective (which, importantly, works on normal varieties) can be found in \cite{rat-maps-normal-vars, truongreldyn}.  The next theorem singles out some properties of the dynamical degrees that we will require.

\begin{theorem}
\label{dyndegstuff}
\leavevmode
\begin{enumerate}
\item\label{dyndegstuff::1} Suppose that \(f : X \rat X\) is birational.  Then \(\lambda_1(f^{-1}) = \lambda_{\dim X - 1}(f)\). Furthermore, if \(\lambda_1(f) > 1\), then \(\lambda_1(f^{-1}) > 1\).
\item\label{dyndegstuff::2} If \(f : X \rat X\) admits an invariant fibration \(\pi : X \to Y\) as above, then \(\lambda_0(f \vert_\pi) = 1\) and
\(
\lambda_1(f)=\max\{\lambda_1(g) , \lambda_1(f\vert_\pi)\}
\).
\item\label{dyndegstuff::3} If \(\dim Y = \dim X - 1\) and \(f\) is birational, then \(g : Y \rat Y\) is birational and \(\lambda_1(f\vert_\pi) = 1\).

\item\label{dyndegstuff::4} Let $f$ (resp.~$g$) be a surjective endomorphism of $X$ (resp.~$Y$) and assume that $X$ and $Y$ are normal projective varieties. If $\pi\colon X\to Y$ is a birational morphism such that $\pi\circ f=g\circ\pi$, then $\lambda_1(f)=\lambda_1(g)$.

\end{enumerate}
\end{theorem}

\begin{proof}
To prove these claims requires using the properties of higher dynamical degrees \(\lambda_p(f)\); since we do not otherwise make use of these degrees, we refer to the above references for the definitions.

The first fact follows from the log-concavity of dynamical degrees, which states that \(\lambda_{p-1}(f) \lambda_{p+1}(f) \leq \lambda_p(f)^2\) for each \(1 \leq p \leq \dim X - 1\).  Since \(\lambda_p(f) \geq 1\) for any \(p\), the hypothesis that \(\lambda_1(f) > 1\) implies that \(\lambda_p(f) > 1\) for each \(1\leq p < \dim X\).  If \(f\) is birational, then \(\lambda_1(f^{-1}) = \lambda_{\dim X - 1}(f) > 1\).

The claim in (\ref{dyndegstuff::2}) that \(\lambda_0(f\vert_\pi) = 1\) follows directly from the definition, while $\lambda_1(f)=\max\{\lambda_1(g)\lambda_0(f\vert_\pi),\lambda_0(g)\lambda_1(f\vert_\pi)\} =\max\{\lambda_1(g) , \lambda_1(f\vert_\pi)\}$ is a case of the product formula of Dinh--Nguy\^en--Truong.

For (\ref{dyndegstuff::3}), another application of the product formula yields \(\lambda_{\dim X}(f) = \lambda_{\dim X - 1}(g) \lambda_1(f\vert_\pi)\).  Since \(f\) is birational, \(\lambda_{\dim X}(f) = 1\), and so both terms on the right must be \(1\) as well.

Finally, (\ref{dyndegstuff::4}) follows from \cite[Theorem 1.(2)]{rat-maps-normal-vars} and the discussion that follows.
\end{proof}

In contrast to the dynamical degrees, the properties of the arithmetic degrees $\overline{\alpha}_f(P)$, $\underline{\alpha}_f(P)$, and $\alpha_f(P)$ are at present largely conjectural in general. 
There is nevertheless a close relationship between the arithmetic and dynamical degrees: it was shown in~\cite[Corollary 9.3]{mss} that if \(f : X \to X\) is a surjective endomorphism with \(\lambda_1(f) > 1\), then there exist points \(P\) with \(\alpha_f(P) = \lambda_1(f)\); however, it remains open whether this equality holds for \emph{every} point $P$ with dense orbit.

\begin{remark}
\label{rmk:lambda1 is 1}
It was proved by Kawaguchi--Silverman \cite[Theorem 4]{dynamical-arithmetic-rat-maps} and Matsuzawa \cite[Theorem 1.4]{on-upper-bounds-arithmetic-degrees} that \(\overline{\alpha}_f(P) \leq \lambda_1(f)\) in general. As a result, the limit defining $\alpha_f(P)$ exists and is equal to $\lambda_1(f)$ if and only if \(\lambda_1(f)\leq \underline{\alpha}_f(P)\); indeed, if this inequality holds, then \(\lambda_1(f) \leq \underline{\alpha}_f(P) \leq \overline{\alpha}_f(P) \leq \lambda_1(f)\).

Furthermore, since we always have \(1 \leq \underline{\alpha}_f(P)\), if \(\lambda_1(f) = 1\), then \(\lambda_1(f) \leq \underline{\alpha}_f(P)\), and so the conjecture holds.  Hence we can always restrict our attention to maps with \(\lambda_1(f) > 1\).
\end{remark}

\label{sec:speculative-results}

We next collect some results in the following general situation: suppose that \(f : X \to X\) is surjective, and there are morphisms \(\pi : X \to Y\) and \(g : Y \to Y\) with \(\pi \circ f = g \circ \pi\).  Under these circumstances, we are in some cases able to reduce the conjecture for \(f\) to the conjecture for \(g\).
There are a number of natural fibrations \(\pi : X \rat Y\) to which one might hope to apply these results on a given variety \(X\), e.g.~the canonical model, the albanese map, Mori fiber spaces, and the mrc quotient. Such canonically defined fibrations play a fundamental role in the study of self-maps of higher-dimensional varieties~\cite{dynamics-auts-zhang}. Recall, as stated in the introduction, that for a regular morphism $f$ and a point $P$ with dense orbit, the limit defining $\alpha_f(P)$ exists, i.e.~$\underline{\alpha}_f(P)=\overline{\alpha}_f(P)$.

\begin{lemma}
\label{l:semi-ample-ht}
Assume that $X$ and $Y$ are normal projective varieties over \(\QQbar\) and let $f$ (resp.~$g$) be a surjective endomorphism of $X$ (resp.~$Y$). If $\pi\colon X\to Y$ is a surjection such that $\pi\circ f=g\circ\pi$ and $P\in X(\QQbar)$ has dense orbit under $f$, then $\alpha_f(P) \geq \alpha_g(\pi(P))$. 
Moreover, if $\pi$ is birational and \(X\) and \(Y\) are \(\QQ\)-factorial, then $\alpha_f(P) = \alpha_g(\pi(P))$.
\end{lemma}
\begin{proof}
We first show $\alpha_f(P) \geq \alpha_g(\pi(P))$. Let $H$ be an ample Cartier divisor on $Y$. Since $P$ has dense orbit under $f$, it follows that $\pi(P)$ has dense orbit under $g$. So, the limit defining $\alpha_g(\pi(P))$ exists and we have 
\[
\alpha_g(\pi(P))=\lim_{n\to\infty} h_H^+(g^n(\pi(P)))^{1/n}=\lim_{n\to\infty} h_{\pi^*H}^+(f^n(P))^{1/n}.
\]
By \cite[Remark 2.2]{mss} (cf. the proof of \cite[Proposition 12]{dynamical-arithmetic-rat-maps}), we obtain
\[
\alpha_f(P)=\overline{\alpha}_f(P)\geq\limsup_{n\to\infty} h_{\pi^*H}^+(f^n(P))^{1/n}=\alpha_g(\pi(P)).
\]
The cited references are formulated under the hypothesis that \(X\) is smooth, so that a dominant rational map induces a pullback map $\phi^*$ on $N^1(X)$.  However, since we assume that \(f\) and \(g\) are regular morphisms, there are no difficulties associated with repeatedly pulling back Cartier divisors, and the same arguments go through on any normal projective variety (see \cite[Remarks 8 and 20]{dynamical-arithmetic-rat-maps}).

It remains to handle the case where $\pi$ is birational. This follows from the proofs of Lemma 3.3 and Theorem 3.4 (ii) in \cite{mss}.  The statement is again formulated under a smoothness hypothesis, but it suffices to assume that \(X\) and \(Y\) are normal and \(\QQ\)-factorial. The negativity lemma holds as long as \(X\) and \(Y\) are normal, and the Weil height machine (\cite[Theorem B.3.2]{hindry-silverman}) remains valid for Cartier divisors on singular varieties by \cite[Remark B.3.2.1]{hindry-silverman}. The \(\QQ\)-factoriality assumption is needed so that the definition of the divisor $E$ in~\cite[Proof of Lemma 3.3]{mss} makes sense: to form \(p^\ast p_\ast q^\ast H_Y\), we must be able to pull back a Weil divisor.
\end{proof}

\begin{corollary}
\label{cor:conj-birational}
Assume that $X$ and $Y$ are normal, \(\QQ\)-factorial projective varieties over \(\QQbar\), and let $f$ (resp.~$g$) be a surjective endomorphism of $X$ (resp.~$Y$). If $\pi\colon X\to Y$ is a birational morphism such that $\pi\circ f=g\circ\pi$, then Conjecture \ref{conj:main} holds for $(X,f)$ if and only if it holds for $(Y,g)$.
\end{corollary}
\begin{proof}
Let $P\in X(\QQbar)$. Then $P$ has dense orbit under $f$ if and only if $\pi(P)$ has dense orbit under $g$. Indeed, since $\pi$ is surjective, it is clear that density of the $f$-orbit of $P$ implies density of the $g$-orbit of $\pi(P)$. Conversely, suppose the $g$-orbit of $\pi(P)$ is dense and let $U\subset X$ be a dense open subset where $\pi\vert_U$ is an isomorphism. Given any open $V\subset X$, we see $V\cap U\neq\varnothing$ and so $\pi(V\cap U)$ contains some $g^n(\pi(P))$. Thus, $V\cap U$ contains $f^n(P)$, proving density of the $f$-orbit of $P$.

To finish the proof, note that Lemma \ref{l:semi-ample-ht} and Theorem \ref{dyndegstuff} (\ref{dyndegstuff::4}) tell us $\alpha_f(P)=\alpha_g(\pi(P))$ and $\lambda_1(f)=\lambda_1(g)$. So, $\alpha_f(P)=\lambda_1(f)$ if and only if $\alpha_g(\pi(P))=\lambda_1(g)$.
\end{proof}

Combining Corollary \ref{cor:conj-birational} with \cite[Theorem 10]{dynamical-equals-arithmetic} yields the following result.

\begin{corollary}
\label{cor:normal-proj-surfaces}
Let $X$ be a normal, \(\QQ\)-factorial projective surface over $\QQbar$. If $f$ is an automorphism of $X$, then Conjecture \ref{conj:main} holds for $(X,f)$.
\end{corollary}
\begin{proof}
Let $\pi\colon\widetilde{X}\to X$ be the minimal resolution. By \cite[Theorem 4-6-2(i)]{matsuki-mmp}, there exists an automorphism $\widetilde{f}$ of $\widetilde{X}$ such that $\eta\circ\widetilde{f}=f\circ\eta$. By \cite[Theorem 2(c)]{dynamical-equals-arithmetic}, Conjecture \ref{conj:main} is known for $(\widetilde{X},\widetilde{f})$ and hence also known for $(X,f)$ by Corollary \ref{cor:conj-birational}.
\end{proof}


\begin{theorem}
\label{thm:fibered-over-dim-one-less}
Let $\pi\colon X\to Y$ be a surjective morphism of normal projective varieties over $\QQbar$.
Suppose $f$ (resp.~$g$) is a surjective endomorphism of $X$ (resp.~$Y$) such that $g\circ\pi = \pi\circ f$. If $\lambda_1(f\vert_\pi) \leq \lambda_1(g)$ 
and Conjecture \ref{conj:main} holds for $(Y,g)$, then Conjecture \ref{conj:main} also holds for $(X,f)$.
The condition \(\lambda_1(f\vert_\pi) \leq \lambda_1(g)\) holds in particular if \(f\) is birational and \(\dim Y = \dim X - 1\).
\end{theorem}
\begin{proof}
We begin by showing that $\lambda_1(g)=\lambda_1(f)$.  By Theorem~\ref{dyndegstuff} (\ref{dyndegstuff::2}), we have
\[
\lambda_1(f)=\max\{\lambda_1(g),\lambda_1(f\vert_\pi)\} = \lambda_1(g).
\]

Next, let $P\in X(\QQbar)$ have dense orbit under $f$, so that $\pi(P)$ has dense orbit under $g$. Then by Lemma \ref{l:semi-ample-ht}, we obtain
\[
\alpha_f(P) \geq \alpha_g(\pi(P)) = \lambda_1(g) = \lambda_1(f),
\]
where $\alpha_g(\pi(P))=\lambda_1(g)$ because the conjecture holds for $(Y,g)$. By Remark \ref{rmk:lambda1 is 1}, we then know that $\alpha_f(P)=\lambda_1(f)$. Therefore, the conjecture holds for $(X,f)$.

Theorem \ref{dyndegstuff} (\ref{dyndegstuff::3}) tells us that \(\lambda_1(f\vert_\pi) = 1\) whenever $f$ is birational and \(\dim Y = \dim X - 1\). Since $\lambda_1(g) \geq 1$, the inequality follows.
\end{proof}

The following consequence of Theorem \ref{thm:fibered-over-dim-one-less} is applied in the proofs of Theorem \ref{thm:intro-CY3-aut} (\ref{intro-CY3-aut::semiample}) and Theorem \ref{thm:intro-MFS-partial} (\ref{intro-MFS-partial::auts3fold}).

\begin{corollary}
\label{cor:fibered-over-surface}
Let $\pi\colon X\to Y$ be a surjective morphism of normal projective varieties over $\QQbar$ with $X$ a threefold and $Y$ a \(\QQ\)-factorial surface. Suppose $f$ (resp.~$g$) is an automorphism of $X$ (resp.~$Y$) such that $g\circ\pi = \pi\circ f$. Then Conjecture \ref{conj:main} holds for $(X,f)$.
\end{corollary}
\begin{proof}
By Corollary \ref{cor:normal-proj-surfaces}, we know Conjecture \ref{conj:main} holds for $(Y,g)$. Since $f$ is birational and $\dim Y = \dim X - 1$, Conjecture \ref{conj:main} for $(X,f)$ follows from Theorem \ref{thm:fibered-over-dim-one-less}.
\end{proof}

\section{Canonical heights for hyper-K\"ahler automorphisms: Theorem \ref{thm:hk}}
\label{sec:HK}
We now treat Conjecture~\ref{conj:main} for hyper-K\"ahler varieties. There are many remarkable automorphisms and birational automorphisms of such varieties, see e.g.~\cite{amerikverbitsky} and \cite{oguiso-hk}. In fact, we prove the conjecture for a wider class of varieties, namely those with trivial Albanese satisfying conditions (A) and (B) of Definition \ref{def:condsAB}. We start by collecting some facts about varieties with numerically trivial canonical class.

\begin{proposition}[{The Beauville--Bogomolov--Fujiki form, see e.g.\ \cite[\S 23]{cybook}, \cite[Prop.\ 25.14]{cybook}}]
\label{bbform}
Suppose that \(X\) is a hyper-K\"ahler variety of dimension \(2m\).  There exists a quadratic form \(q(X)\) on \(H^2(X,\RR)\) and a constant \(c_X\) such that for any divisor \(D\), we have:
\[
D^{2m} = c_X \, q_X(D)^m.
\]
The form \(q_X(-)\) has signature \((1,\rho(X)-1)\) on \(N^1(X)_\RR\).
If \(\phi : X \to X\) is an automorphism, then the pullback \(\phi^\ast\) preserves the form \(q(-)\).
\end{proposition}

\begin{example}[{\cite[\S 6]{beauville}, \cite{oguiso-hk}}]
A basic example of a hyper-K\"ahler manifold is the Hilbert scheme of configurations of \(n\) points on a K3 surface \(S\), which we denote by \(\Hilb^n(S)\).  If \(f : S \to S\) is an automorphism, then the induced automorphism \(f^{[n]} : \Hilb^n(S) \to \Hilb^n(S)\) of the Hilbert scheme satisfies \(\lambda_1(f^{[n]}) = \lambda_1(f)\).
\end{example}

To begin, we first note that any surjective endomorphism of a hyper-K\"ahler variety is actually an automorphism.

\begin{lemma}
\label{l:HK all endos are autos}
Every surjective endomorphism of a hyper-K\"ahler variety over \(\QQbar\) is an automorphism.
\end{lemma}
\begin{proof}
By fpqc descent, it is enough to check this after base change to $\CC$. 
Now if \(X\) is hyper-K\"ahler variety over $\CC$, we have \(\chi(\O_X) = 1+\frac{1}{2}\dim X\) \cite[Lemma 14.21]{Einstein-mfds}.
Next, let $f$ be a surjective endomorphism of $X$. By \cite[Lemma 2.3]{fujimoto}, $f$ is a finite \'etale cover and so $\chi(\O_X)=\deg(f)\chi(\O_X)$. Since $\chi(\O_X)\neq0$, we must have $\deg(f)=1$, i.e.~$f$ is an automorphism.
\end{proof}

Having now reduced to the case of automorphisms, we roughly follow the strategy of Kawaguchi's proof of the conjecture in the case of surfaces.   The following version of the Perron--Frobenius theorem plays an important role.

\begin{lemma}[\cite{birkhoff}]
\label{l:Birkoff}
Suppose that \(V\) is a finite-dimensional real vector space and that \(K \subset V\) is a closed, pointed, convex cone.  If \(T : V \to V\) is a linear map for which \(T(K) \subseteq K\), and the spectral radius of \(T\) is \(\lambda > 1\), then there exists a \(\lambda\)-eigenvector for \(T\) which is contained in \(K\).
\end{lemma}

\begin{definition}
We will say that a class \(\nu_+\) in \(N^1(X)_\RR\) is a \emph{leading eigenvector} for \(f\) if it is a nef class which is a \(\lambda_1(f)\)-eigenvector for \(f^\ast\).  We say that \((\nu_+,\nu_-)\) is an \emph{eigenvector pair} for \(f\) if \(\nu_+\) is a leading eigenvector for \(f\) and \(\nu_-\) is a leading eigenvector for \(f^{-1}\).

We say that \((D_+,D_-)\) is an \emph{eigendivisor pair} for \(f\) if \(D_+\) and \(D_-\) are \(\RR\)-divisors for which \(f^\ast(D_+) \sim_{\RR} \lambda_1(f) D_+\) and \((f^{-1})^\ast(D_-) \sim_{\RR} \lambda_1(f^{-1}) D_-\).  If \((D_+,D_-)\) is an eigendivisor pair, then the corresponding pair of numerical classes \((\nu_+,\nu_-)\) is an eigenvector pair.
\end{definition}

\begin{corollary}
\label{cor:eigenvector pairs exist}
If $f$ is an automorphism of a normal projective variety $X$ which satisfies \(\lambda_1(f) > 1\), then an eigenvector pair $(\nu_+,\nu_-)$ exists for $f$.  If moreover \(h^1(X,\O_X) = 0\), then an eigendivisor pair \((D_+,D_-)\) exists for \(f\).
\end{corollary}

\begin{proof}
Recall that \(\lambda_1(f^{-1}) = \lambda_{\dim X-1}(f) > 1\) by Theorem~\ref{dyndegstuff} (\ref{dyndegstuff::1}).  We obtain the result by applying Lemma \ref{l:Birkoff} to the case where \(V = N^1(X)_{\RR}\), \(K = \Nef(X)\), and \(T\) is the pullback \(f^\ast : N^1(X)_\RR \to N^1(X)_\RR\) or the pullback $(f^{-1})^\ast$.  Note in particular that \(\nu_+\) and \(\nu_-\) belong to the cone \(\Nef(X)\).

If \(h^1(X,\O_X) = 0\), then the map \(\Pic(X)_{\RR}\to N^1(X)_\RR\) is an isomorphism, and we may take \(D_\pm\) to be the unique lift of \(\nu_\pm\) to a linear equivalence class.
\end{proof}

We next single out three special properties of hyper-K\"ahler manifolds and their automorphisms. We will prove Conjecture \ref{conj:main} for any automorphisms satisfying these properties, which includes some non-hyper-K\"ahler examples as well.

\begin{definition}
\label{def:condsAB}
Suppose that \(f : X \to X\) is an automorphism of a normal projective variety \(X\).  We say that \(f\) has property
\begin{enumerate}
\item[(A)] if \(\lambda_1(f) > 1\) and \(\lambda_1(f) = \lambda_1(f^{-1})\);
\item[(B)] if \(\nu = \nu_+ + \nu_-\) is big for some eigenvector pair \((\nu_+,\nu_-)\);
\item[(C)] if \(h^1(X,\O_X) = 0\).
\end{enumerate}
\end{definition}

Recall that by Remark \ref{rmk:lambda1 is 1}, the conjecture is known whenever $\lambda_1(f)=1$. So there is never any harm in assuming $\lambda_1(f)>1$.

\begin{lemma}
\label{aandbhold}
Let \(f : X \to X\) be an automorphism of a normal projective variety $X$ and assume that $\lambda_1(f)>1$.  Then Conditions (A) and (B) hold in all of the following cases:
\begin{enumerate}
\item the dimension of \(X\) is equal to \(2\);
\item the Picard rank of \(X\) is equal to \(2\);
\item \(X\) is a hyper-K\"ahler variety.
\end{enumerate}
\end{lemma}

\begin{proof}

In dimension \(2\), both conditions (A) and (B) are well-known consequences of the Hodge index theorem (see e.g.~\cite[Proposition 2.5]{surface-aut}).

Suppose instead that \(\rho(X) = 2\).  The pullback \(f^\ast : N^1(X) \to N^1(X)\) is invertible and so has determinant \(\pm 1\); letting the eigenvalues be $a$ and $b$, we then have $ab=\pm1$ where say $\abs{a}\geq \abs{b}$. Then $1<\lambda_1(f)=\abs{a}$. So the eigenvalues of $(f^{-1})^\ast$ are $a^{-1}=\pm b$ and $b^{-1}=\pm a$, and we see $\lambda_1(f^{-1})=\abs{\pm a}=\lambda_1(f)$. This verifies Condition (A). Since the Picard rank is \(2\), the sum of the two classes on the boundary of the pseudoeffective cone is big; in fact, if \(\nu_+\) and \(\nu_-\) are normalized appropriately, the sum may be assumed ample, which yields Condition (B).

We come at last to the hyper-K\"ahler case; the argument is the same as that in the two-dimensional setting, but with the Beauville--Bogomolov form standing in for the usual intersection product.  To verify Condition (A), we use a result of Oguiso~\cite[Theorem 1.1]{oguiso-hk} which tells us $\lambda_1(f)=\lambda_{\dim X-1}(f)$. Then by Theorem~\ref{dyndegstuff} (\ref{dyndegstuff::1}), we have \(\lambda_1(f^{-1}) = \lambda_{\dim X-1}(f) > 1\). We next check Condition (B).  Let \(\dim X = 2m\) and $(\nu_+,\nu_-)$ be an eigenvector pair for $f$, whose existence is guaranteed by Corollary \ref{cor:eigenvector pairs exist}; define $\nu =\nu_+ + \nu_-$. Since \(\nu\) is nef, its volume can be computed as the top self-intersection \(\nu^{2k}\), and \(\nu\) is big if and only if this number is positive.
We have
\[
q(\nu_+) = q(f^\ast \nu_+) = \lambda_1(f)^2 q(\nu_+),
\]
and so \(q(\nu_+) = 0\) since \(\lambda_1(f) > 1\).  The same argument shows that \(q(\nu_-) = 0\).
Since the form \(q_X(-)\) has signature \((1,\rho(X)-1)\) on \(\Pic(X)\), the maximal dimension of an isotropic subspace is \(1\), and so \(q_X(\nu) \neq 0\).  Since \(\nu\) is nef, \(\Vol(\nu) = \nu^{2m} = c_X q_X(\nu)^m > 0\), and we conclude that \(\nu\) is big.
\end{proof}

\begin{remark}
Notice that if a variety $X$ has Picard rank \(2\) and an automorphism $f$ with \(\lambda_1(f) > 1\), then necessarily \(K_X \equiv 0\), since otherwise \(K_X\) gives a \(1\)-eigenvector.  However, there are many interesting examples in this case~\cite{oguisorhois2,dqzhangrhois2}.
\end{remark}

\begin{lemma}
\label{l:CondB->middle-coeff}
Suppose that \(f : X \to X\) is an automorphism of a normal projective variety satisfying Condition (B) for some eigenvector pair $(\nu_+,\nu_-)$.  There is a unique value \(0 < \ell < \dim X\) for which \(\nu_+^\ell \cdot \nu_-^{\dim X-\ell}\) is nonzero, and
\[
\lambda_1(f)^\ell = \lambda_1(f^{-1})^{\dim X-\ell}.
\]
\end{lemma}

\begin{proof}
Since \(\nu_+ + \nu_-\) is nef, its volume is equal to the top self-intersection \(( \nu_+ + \nu_- )^{\dim X}\). Since volume is preserved by pullback, the following holds for all positive integers $m$.
\begin{align*}
  ( \nu_+ + \nu_- )^{\dim X} &= (f^\ast)^m ( \nu_+ + \nu_- )^{\dim X} = (\lambda_1(f)^m \nu_+  + \lambda_1(f^{-1})^{-m} \nu_- )^{\dim X} \\
  &= \sum_{j=0}^{\dim X} \binom{\dim X}{j} (\lambda_1(f)^{j} \lambda_1(f^{-1})^{j-\dim X})^m \nu_+^j \cdot \nu_-^{\dim X-j}.
\end{align*}
In order for this quantity to be independent of \(m\), there must be at most one nonzero term \(j=\ell\), and the coefficient \(\lambda_1(f)^{\ell} \lambda_1(f^{-1})^{\ell-\dim X}\) must equal \(1\).
\end{proof}

\begin{remark}
If \(\dim X\) is odd, then an automorphism satisfying Condition (B) cannot satisfy Condition (A); indeed, Lemma \ref{l:CondB->middle-coeff} tells us we need $\lambda_1(f)^{\dim X - 2\ell}=1$ which is impossible. 
If \(\dim X\) is even and $f$ satisfies Condition (B), then $f$ satisfies Condition (A) if and only if \(\ell = \frac{1}{2}\dim X\).  
\end{remark}

We now turn to the proof of the Kawaguchi--Silverman conjecture in this setting.  The strategy is to construct a canonical height function for the automorphism $f$, following the work of Silverman~\cite{silvermank3} and Kawaguchi~\cite{surface-aut}, together with some inputs from birational geometry.

\begin{definition}
\label{def:augmented base locus}
Suppose that \(D\) is a \(\QQ\)-divisor on a normal projective variety \(X\).  The \emph{stable base locus} of \(D\) is the Zariski-closed subset of \(X\) defined by
\[
\BB(D) = \kern-1.2em \bigcap_{\substack{m \geq 1 \\ \text{$mD$ Cartier}}} \kern -1.1em \Bs(mD).
\]
\end{definition}

It is not hard to show that there exists an integer \(m_0\) such that \(\Bs(dm_0D) = \BB(D)\) for all sufficiently large integers \(d\)~\cite[Proposition 2.1.21]{lazarsfeld}. It follows that $\BB(D+D')\subset \BB(D)\cup\BB(D')$.

Suppose that \(D\) is an \(\RR\)-divisor on a normal projective variety \(X\).   The \emph{augmented base locus} \(\BB_+(D)\) is the Zariski-closed subset
\[
\BB_+(D) = \kern-1.8em\bigcap_{\substack{\text{$A$ ample} \\ \text{$D-A$ $\QQ$-divisor}}} \kern-1.6em \BB(D-A).
\]

We refer to \cite{elmnp1} for a detailed treatment of the properties of the invariant \(\BB_+(D)\), but single out the following:

\begin{lemma}[{\cite[Prop.\ 1.4, Examples 1.7--1.9, Prop.\ 1.5]{elmnp1}}] \label{bplusproperties}
\leavevmode
\begin{enumerate}
\item\label{bplusproperties::1} \(\BB_+(D)\) depends only on the numerical class of \(D\).
\item\label{bplusproperties::2} \(\BB_+(D)\) is a proper subset of \(X\) if and only if \(D\) is big.
\item\label{bplusproperties::3} For any \(\RR\)-divisor \(D\) and any real \(\lambda > 0\), we have \(\BB_+(D) = \BB_+(\lambda D)\).
\item\label{bplusproperties::4} For any \(\RR\)-divisors \(D_1\) and \(D_2\), we have \(\BB_+(D_1 + D_2) \subseteq \BB_+(D_1) \cup \BB_+(D_2)\).
\item\label{bplusproperties::5} Fix a norm \(\dabs{\cdot}\) on \(N^1(X)_\RR\).  For any \(\RR\)-divisor \(D\), there exists a constant \(\epsilon\) such that for any ample \(\RR\)-divisor \(A\) for which \(\dabs{A} < \epsilon\) and \(D-A\) is a \(\QQ\)-divisor, we have \(\BB_+(D) = \BB(D-A)\).
\end{enumerate}
\end{lemma}

In view of (\ref{bplusproperties::1}), we sometimes write \(\BB_+(\nu)\) where \(\nu\) is any class in \(N^1(X)_\RR\); this denotes \(\BB_+(D)\) for any \(D\) with numerical class \(\nu\).

\begin{lemma}
\label{bplussum}
Suppose that \(D_1\) is a big \(\RR\)-divisor and \(D_2\) is a nef \(\RR\)-divisor. Then
\[
\BB_+(D_1+D_2) \subseteq \BB_+(D_1).
\]
\end{lemma}

\begin{proof}
First, choose an ample \(\RR\)-divisor \(A_1\) so that \(D_1 - A_1\) is a \(\QQ\)-divisor and
\[
\BB_+(D_1) = \BB(D_1 - A_1).
\]
Now, choose another ample \(\RR\)-divisor \(A_2\) for which
\(D_1 + D_2 - A_2\) is a \(\QQ\)-divisor, \(A_1 - A_2\) is ample, and
\[
\BB_+(D_1 + D_2) = \BB(D_1 + D_2 - A_2).
\]

It again follows from Lemma~\ref{bplusproperties} (\ref{bplusproperties::5}) that \(A_2\) may be taken to be any sufficiently small ample divisor for which \(D_1 + D_2 - A_2\) is a \(\QQ\)-divisor.  Note that the divisor \(D_2 + (A_1 - A_2) = (D_1 + D_2 - A_2) - (D_1 - A_1)\) is again a \(\QQ\)-divisor, and we may then compute:
\begin{align*}
  \BB_+(D_1 + D_2) &= \BB(D_1 + D_2 - A_2)
  = \BB \left( (D_1 - A_1) + (D_2 + (A_1 - A_2)) \right) \\
  &\subseteq \BB(D_1 - A_1) \cup \BB( D_2 + (A_1 - A_2) ) \\
  &= \BB_+(D_1) \cup \BB(D_2 + (A_1 - A_2)) = \BB_+(D_1),
\end{align*}
where \(\BB(D_2 + (A_1 - A_2))\) is empty since \(D_2\) is nef and \(A_1 - A_2\) is ample.
\end{proof}

\begin{lemma}
\label{bpluscombos}
Suppose that \(D_1\) and \(D_2\) are nef \(\RR\)-divisors.  Then for any \(a_1,a_2 > 0\), the locus \(\BB_+(a_1 D_1 + a_2 D_2)\) is independent of \(a_1\) and \(a_2\).
\end{lemma}

\begin{proof}
We show that for any \(a_1,a_2 > 0\), we have \(\BB_+(a_1 D_1 + a_2 D_2) = \BB_+(D_1 + D_2)\). Suppose first that \(a_1 \geq a_2\).  Recalling that \(\BB_+(D) = \BB_+(\lambda D)\) according to Lemma~\ref{bplusproperties} (\ref{bplusproperties::3}), it follows from Lemma~\ref{bplussum} that
\begin{align*}
  \BB_+(a_1 D_1 + a_2 D_2) &= \BB_+(a_2(D_1 + D_2) + (a_1 - a_2) D_1) \\ &\subseteq \BB_+(a_2(D_1 + D_2)) = \BB_+(D_1 + D_2), \\
  \BB_+(D_1 + D_2) &= \BB(a_1(D_1+D_2)) = \BB( (a_1 D_1 + a_2 D_2) + (a_1 - a_2) D_2) \\ &\subseteq \BB_+(a_1 D_1 + a_2 D_2).
\end{align*}
The case when \(a_1 < a_2\) follows from the same argument, reversing the roles of \(D_1\) and \(D_2\).
\end{proof}

\begin{corollary}
\label{cor:invariant-B+}
Suppose that \(f : X \to X\) is an automorphism of a normal projective variety with \(\lambda_1(f) > 1\), and let \((\nu_+,\nu_-)\) be an eigenvector pair.  Then \(\BB_+(\nu_++\nu_-)\) is invariant under \(f\). Furthermore, if \(f\) satisfies Condition (B) and \(P\) is a $\QQbar$-point of \(X\) with Zariski-dense orbit under \(f\), then \(P\) is not contained in \(\BB_+(\nu_+ + \nu_-)\).
\end{corollary}
\begin{proof}
We have
\begin{align*}
f(\BB_+(\nu_++\nu_-)) &= \BB_+((f^{-1})^\ast(\nu_++\nu_-)) \\ &=  \BB_+(\lambda_1(f)^{-1} \nu_+ + \lambda_1(f^{-1}) \nu_-) = \BB_+(\nu_++\nu_-),
\end{align*}
where the final equality follows from Lemma~\ref{bpluscombos}.

If \(f\) satisfies Condition (B), then \(\nu = \nu_+ + \nu_-\) is big, and so \(\BB_+(\nu)\) is a proper Zariski-closed subset of \(X\), invariant under \(f\).  It follows that a point with dense orbit cannot lie in \(\BB_+(\nu)\).
\end{proof}

\begin{remark}
In the case \(\dim X = 2\), it follows from~\cite[Proposition 3.1(2)]{surface-aut} that the locus \(\BB_+(\nu)\) is precisely the union of the \(f\)-invariant curves.  If \(\rho(X) = 2\), then \(\nu\) may be assumed ample after a suitable choice of normalization for \(\nu_+\), and \(\BB_+(\nu)\) is empty.
\end{remark}

\begin{proposition}[{The Weil height machine, e.g.\ \cite[Theorem B.3.6]{hindrysilverman}}]
Let \(X\) be a projective variety defined over \(\QQbar\). There exists a unique map
\[
\Pic(X)_{\RR} \to \frac{ \set{ \text{functions $X(\QQbar) \to \RR$} }}{ \set{ \text{bounded functions $X(\QQbar) \to \RR$}}}
\]
with the following properties:
\begin{enumerate}
\item Normalization: if \(D\) is very ample, $\phi_D\colon X\to\PP^n$ is the associated embedding, and $h$ is the absolute logarithmic height \cite[\S B.2]{hindrysilverman}, then \(h_D(P) = h(\phi_D(P)) + O(1)\).
\item Functoriality: if \(\pi : X \to Y\) is a morphism, then \(h_{X,\pi^\ast D}(P) = h_{Y,D}(\pi(P)) + O(1)\).
\item Additivity: \(h_{X,D_1+D_2}(P) = h_{X,D_1}(P) + h_{X,D_2}(P) + O(1)\)
\item Positivity: If \(D\) is effective, then \(h_{X,D}(P) \geq O(1)\) for \(P\) outside the base locus of \(D\).
\end{enumerate}
\end{proposition}

By \emph{a height function} for an \(\RR\)-divisor class \(D\), we mean a function \(h_D : X(\QQbar) \to \RR\) belonging to the class of height functions for \(D\).

The augmented base locus is well-suited to working with height functions associated to big \(\RR\)-divisors.  The next two lemmas give extensions of the positivity property and Northcott's lemma to this setting.

\begin{lemma}
\label{positivity}
Let \(X\) be a normal, projective variety over \(\QQbar\).
\begin{enumerate}
\item\label{positivity:1} Suppose that \(D\) is a \(\QQ\)-divisor.  Then \(h_{X,D}(P) \geq O(1)\) for \(P\) outside \(\BB(D)\).   
\item\label{positivity:2} Suppose that \(D\) is a \(\RR\)-divisor.  Then \(h_{X,D}(P) \geq O(1)\) for \(P\) outside \(\BB_+(D)\).
\item\label{positivity:3} Suppose that \(D\) is a big \(\RR\)-divisor on \(X\).  Then for any \(M\) and \(N\), there are only finitely many points \(P\) of \(X(\QQbar) \setminus  \BB_+(D)\) with \([\QQ(P):\QQ] < M\) and \(h_{D}(P) < N\).
\end{enumerate}
\end{lemma}

\begin{proof}
Fix an integer \(m\) with \(\Bs(mD) = \BB(D)\); then \(h_{X,mD} = m \, h_{X,D} + O(1)\) according to the additivity property, and (\ref{positivity:1}) follows.

For (\ref{positivity:2}), according to Lemma \ref{bplusproperties} (\ref{bplusproperties::5}) there exists an ample \(\RR\)-divisor \(A\) so that \(D-A\) is a \(\QQ\)-divisor and \(\BB_+(D) = \BB(D-A)\).  According to (\ref{positivity:1}), we have \(h_{D-A}(P) \geq O(1)\) for \(P\) outside \(\BB(D-A) = \BB_+(D)\).  Since \(h_D = h_{D-A} + h_A + O(1)\), and since \(h_A \geq O(1)\), this proves (\ref{positivity:2}).

At last we prove (\ref{positivity:3}); let \(D\) and \(A\) be as before.  There is a constant \(C_1\) such that \(h_A(P) \leq h_{D}(P) - h_{D-A}(P) + C_1\) for all points \(P\) of \(X(\QQbar)\).  By (\ref{positivity:1}), there is a constant \(C_2\) so that \(h_{D-A}(P) \geq C_2\) for any \(P\) in \(X(\QQbar) \setminus \BB(D-A) = X(\QQbar) \setminus \BB_+(D)\).
Now, if \(P\) is a point of \(X(\QQbar) \setminus \BB_+(D)\) with \(h_D(P) < N\), we have
\[
h_A(P) \leq h_D(P) - h_{D-A}(P) + C_1 \leq h_D(P) + C_1 - C_2 \leq N + C_1 - C_2.
\]

It then follows from the Northcott theorem for the ample divisor \(A\) that there are only finitely many such $P$ with $[\QQ(P):\QQ]<M$ and \(h_D(P) < N\), see Theorem B.3.2(g) and Remark B.3.2.1(i) of~\cite{hindrysilverman}.
\end{proof}

With these results in place, we now construct a canonical height function for an automorphism satisfying Conditions (A) and (B) and which admits an eigendivisor pair.  Suppose that \(f : X \to X\) is an automorphism of a normal projective variety satisfying these  conditions, with \((D_+,D_-)\) an eigendivisor pair for \(f\).
Define functions \(\widehat{h}_{D_+}: X(\overline{\QQ}) \to \RR\) and \(\widehat{h}_{D_-}: X(\overline{\QQ}) \to \RR\) by
\begin{align*}
\widehat{h}_{D_+}(P) &= \lim_{n \to \infty} \frac{1}{\lambda_1(f)^n} h_{D_+}(f^n(P)) \\
\widehat{h}_{D_-}(P) &= \lim_{n \to \infty} \frac{1}{\lambda_1(f^{-1})^n} h_{D_-}(f^{-n}(P))
\end{align*}
The functoriality of the height function yields \(h_{D_\pm}(P) - \lambda_1(f^{\pm 1})^{-1} h_{D_\pm}(f(P)) = O(1)\);  
it follows from an argument of Tate (cf.~\cite[\S 3]{silvermank3}) that both of these limits exist and that \(\widehat{h}_{D_\pm}\) is a height function for \(D_\pm\).
These functions furthermore satisfy the relations
\[
\widehat{h}_{D_+}(f(P)) = \lambda_1(f) \widehat{h}_{D_+}(P), \quad
\widehat{h}_{D_-}(f(P)) = \lambda_1(f^{-1})^{-1} \widehat{h}_{D_-}(P).
\]
Consider the function \(\widehat{h}(P) : X(\overline{\QQ}) \to \RR\) given by
\[
\widehat{h}(P) = \widehat{h}_{D_+}(P) + \widehat{h}_{D_-}(P)
\]
We next develop the properties of \(\widehat{h}(P)\), closely following arguments of Kawaguchi~\cite[Theorem 5.2 and Proposition 5.5]{surface-aut}.

\begin{theorem}
\label{heightfct}
Let $X$ be a normal projective variety over $\QQbar$. Let $f$ be an automorphism of $X$ satisfying Conditions (A) and (B), and suppose that \(f\) admits an eigendivisor pair $(D_+,D_-)$. If $P\in X(\QQbar)$, then the function \(\widehat{h}\) has the following properties: 
\begin{enumerate}
\item\label{heightfct:1} \(\widehat{h}(f^n(P)) + \widehat{h}(f^{-n}(P)) = (\lambda_1(f)^n+ \lambda_1(f)^{-n}) \widehat{h}(P)\) for any $n$;
\item\label{heightfct:2} \(\widehat{h}\) is a height function for the big and nef divisor \(D = D_+ + D_-\);
\item\label{heightfct:3} \(\widehat{h}(P) \geq 0\) if \(P\in X(\QQbar) \setminus \BB_+(D)\);
\item\label{heightfct:4} \(\widehat{h}_{D_+}(P) \geq 0\) and \(\widehat{h}_{D_-}(P) \geq 0\) if \(P\in X(\QQbar) \setminus \BB_+(D)\);
\item\label{heightfct:6} \(\widehat{h}\) satisfies the Northcott property on \(X(\QQbar) \setminus \BB_+(D)\);
\item\label{heightfct:7} If \(P\in X(\QQbar) \setminus \BB_+(D)\), then \(\widehat{h}_{D_+}(P) = 0\) if and only if \(\widehat{h}_{D_-}(P) = 0\) if and only if \(\widehat{h}(P) = 0\) if and only if \(P\) is \(f\)-periodic.
\end{enumerate}
\end{theorem}

\begin{proof}
Let $\lambda=\lambda_1(f)$ and let $D=D_++D_-$, which is big by Condition (B) and nef since it is a sum of two nef classes. By Condition (A), we know $\lambda_1(f^{-1})=\lambda$. It then follows directly from the definitions that $\widehat{h}(f^\pm(P))=\lambda^\pm\widehat{h}_{D_+}(P)+\lambda^{\mp}\widehat{h}_{D_-}(P)$; property (\ref{heightfct:1}) then follows.

(\ref{heightfct:2}) is a consequence of the fact that \(\widehat{h}_{D_+}\) and \(\widehat{h}_{D_-}\) are themselves height functions for \(D_+\) and \(D_-\).

For (\ref{heightfct:3}), Lemma \ref{positivity} (\ref{positivity:2}) implies that there is a constant \(C\) so that \(\widehat{h}(P) \geq C\) for any \(P\) in \(X(\QQbar) \setminus \BB_+(D)\).  Furthermore, for any such point \(P\), Corollary \ref{cor:invariant-B+} shows that \(f^n(P)\in X(\QQbar) \setminus \BB_+(D)\) for all $n\in\ZZ$, so that \(\widehat{h}(f^n(P)) \geq C\)  for all $n$. 
Then by property (\ref{heightfct:1}), for any integer \(n\),
\[
\widehat{h}(P) = \frac{1}{\lambda^n + \lambda^{-n}} \left( \widehat{h}(f^n(P)) + \widehat{h}(f^{-n}(P)) \right) \geq \frac{1}{\lambda^n + \lambda^{-n}} (2 C),
\]
and the non-negativity of \(\widehat{h}(P)\) follows by taking \(n \to \infty\).  This finishes (\ref{heightfct:3}).

Now, since \(\widehat{h}(P) \geq 0\) for any \(P\in X(\QQbar) \setminus\BB_+(D)\), we have $\widehat{h}_{D_+}(P) \geq -\widehat{h}_{D_-}(P)$ for all such $P$. So,
\[
\widehat{h}_{D_+}(P) = \lambda^{-n} \widehat{h}_{D_+}(f^n(P)) \geq -\lambda^{-n} \widehat{h}_{D_-}(f^n(P)) = -\lambda^{-2n} \widehat{h}_{D_-}(P),
\]
and the non-negativity of \(\widehat{h}_{D_+}(P)\) follows by taking the limit as \(n\) tends to infinity; non-negativity of \(\widehat{h}_{D_-}(P)\) follows from a similar argument. This handles (\ref{heightfct:4}).

Lemma \ref{positivity} (\ref{positivity:3})~combined with property (\ref{heightfct:2}) immediately implies (\ref{heightfct:6}).

We now turn to (\ref{heightfct:7}). First, if $P$ is $f$-periodic, then \(f^n(P) = P\) for some $n$, which implies directly from the definitions that \(\widehat{h}_{D_+}(P)\) and \(\widehat{h}_{D_-}(P)\) both vanish, and hence \(\widehat{h}(P) = 0\).  On the other hand, suppose that \(\widehat{h}(P) = 0\) for some \(P\) in \(X(\QQbar) \setminus \BB_+(D)\). Then (\ref{heightfct:1}) tells us $\widehat{h}(f^n(P)) + \widehat{h}(f^{-n}(P))=0$ for all $n$. Since $\widehat{h}(f^n(P))$ and $\widehat{h}(f^{-n}(P))$ are non-negative by (\ref{heightfct:3}), we see \(\widehat{h}(f^n(P)) = 0\) for all \(n\). By Corollary \ref{cor:invariant-B+}, the locus $\BB_+(D)$ is $f$-invariant, so $\{f^n(P)\mid n\in\ZZ\}$ is contained in $X(\QQbar) \setminus\BB_+(D)$. Since the $f^n(P)$ are of bounded degree over \(\QQ\), the Northcott property (\ref{heightfct:6}) tells us the set of \(f^n(P)\) is finite, and so \(P\) is $f$-periodic.

To finish the proof of (\ref{heightfct:7}), it remains to show that if $P\in X(\QQbar) \setminus \BB_+(D)$ and $h_{D_+}(P)=0$, then $P$ is $f$-periodic; the assertion that $h_{D_-}(P)=0$ implies $P$ is $f$-periodic will follow similarly. If $\widehat{h}_{D_+}(P)=0$ and $n>0$ then $\widehat{h}(f^n(P))=\widehat{h}_{D_+}(f^n(P))+\widehat{h}_{D_-}(f^n(P))=\lambda^n\widehat{h}_{D_+}(P)+\lambda^{-n}\widehat{h}_{D_-}(P)=\lambda^{-n}\widehat{h}_{D_-}(P)\leq\widehat{h}_{D_-}(P)$. Since we have fixed our point $P$, we can view $\widehat{h}_{D_-}(P)$ as a constant and we have bounded $\widehat{h}(f^n(P))$ for all $n$. Again, since the set $\{f^n(P)\}$ has bounded degree and is contained in $X(\QQbar) \setminus\BB_+(D)$, it follows from (\ref{heightfct:6}) that $\{f^n(P)\}$ is finite, and so $P$ is $f$-periodic.
\end{proof}

\begin{theorem}
\label{ksforaandb}
Suppose that \(f : X \to X\) is an automorphism of a normal projective variety satisfying conditions (A), (B), and (C).  Then the Kawaguchi--Silverman conjecture holds for \(f\).
\end{theorem}
\begin{proof}
By Corollary~\ref{cor:eigenvector pairs exist}, there exists an eigendivisor pair \((D_+,D_-)\), with \(D = D_+ + D_-\) big.
Let $P\in X(\QQbar)$ have dense orbit under $f$. By Corollary \ref{cor:invariant-B+}, we know $P$ does not lie in \(\BB_+(D)\). Since $P$ is not $f$-periodic, we know from Theorem~\ref{heightfct} (\ref{heightfct:7}) that \(\widehat{h}_{D_+}(P)\) and \(\widehat{h}_{D_-}(P)\) are both positive. Then
\[
\alpha_f(P)\geq\liminf_{n\to\infty}h_D^+(f^n(P))^{1/n}=\liminf_{n\to\infty}(\lambda_1(f)^n\widehat{h}_{D_+}(P)+\lambda_1(f)^{-n}\widehat{h}_{D_-}(P))^{1/n}=\lambda_1(f),
\]
where the inequality follows from \cite[Remark 2.2]{mss} and the next equality from Theorem \ref{heightfct} (\ref{heightfct:2}) which tells us that $h_D=\widehat{h}$.
\end{proof}

It follows from Lemma~\ref{aandbhold} and Theorem~\ref{ksforaandb} that the Kawaguchi--Silverman conjecture holds for automorphisms of hyper-K\"ahler varieties; note that Condition (C) holds since a hyper-K\"ahler variety is geometrically simply connected.  The next lemma shows that it also holds for automorphisms of smooth varieties of Picard rank \(2\), slightly extending~\cite[Theorem 4.2(ii)]{shibatapic2}.

\begin{theorem}
Suppose that \(X\) is a smooth projective variety, \(\rho(X) = 2\) and \(f : X \to X\) is an automorphism.  Then the Kawaguchi--Silverman conjecture holds for \(f\).
\end{theorem}

\begin{proof}
The map \(\phi\) induces an automorphism \(g : \Alb(X) \to \Alb(X)\) such that the diagram
\[
\xymatrix{
X \ar[r]^f \ar[d]^a & X \ar[d]^a \\
\Alb(X) \ar[r]^g & \Alb(X)
}
\]
commutes.  The proof of Lemma~\ref{aandbhold} shows that neither eigenvalue of \(f^\ast : N^1(X)_\RR \to N^1(X)_\RR\) is equal to \(1\), and so it must be that \(K_X \equiv 0\).  A form of abundance due to Nakayama~\cite{nakayama} implies that \(K_X\) is torsion in \(\Pic(X)\), so that \(\kappa(X) = 0\).  Since \(\kappa(X) = 0\), a result of Kawamata (independent of the conjectures of the MMP) implies that \(a\) is surjective with connected fibers~\cite{kawamataabelian}.

If \(\dim \Alb(X) = 0\), then \(h^1(X,\O_X) = 0\), so that Condition (C) is satisfied.  In this case, Conjecture~\ref{conj:main} follows from Theorem~\ref{ksforaandb}.
If \(a\) is finite, it must be an isomorphism, since \(a\) has connected fibers.  Then \(X\) is an abelian variety and the conjecture holds by~\cite{silverman-ab-vars}.

Suppose at last that \(a\) is not finite and that \(\dim \Alb(X) > 0\).  It must be that \(\rho(X) \geq \rho(\Alb(X))+1\), since for any divisor \(D\) on \(Y\), \(\pi^\ast D\) has intersection \(0\) with a curve in the fiber of \(a\).  Since \(\rho(X) = 2\), we have \(\rho(\Alb(X)) = 1\).  Taking \(H\) to be a generator of $\Pic(\Alb(X))$, it must be that \(a^\ast H\) is a \(1\)-eigenvector for \(\phi^\ast\), but neither eigenvalue of \(\phi^\ast\) is equal to \(1\), so this case is impossible.
\end{proof}

Notice that if \(\rho(X) = 2\) and \(h^1(X,\O_X) = 0\), we have proved something even stronger: since \(D = D_+ + D_-\) may be taken to be ample, \(\BB_+(D) = \emptyset\), and so \(\alpha_f(P) = \lambda_1(f)\) for every \(\QQbar\)-point \(P\), without assuming the orbit is Zariski-dense.

Now, suppose that \(f : X \to X\) is an automorphism of a normal, projective variety satisfying \(\lambda_1(f) > 1\), and that \(V\) is an irreducible subvariety of \(f\) periodic under \(f\). Then \(f^n\) maps \(V\) to itself, and so induces an automorphism \(\widetilde{f^n} : \widetilde{V} \to \widetilde{V}\), where \(\widetilde{V}\) is the normalization of \(V\).  Then set
\[
\lambda_1(f \vert_V) = \lambda_1(\widetilde{f^n} \vert_{\widetilde{V}})^{1/n}.
\]

\begin{lemma}
\label{lem:autorestrict}
Let \(f : X \to X\) be an automorphism of a normal projective variety satisfying Condition (B), and let \((\nu_+,\nu_-)\) be an eigenvector pair.  If \(\lambda_1(f\vert_V) < \lambda_1(f)\), then \(\nu_+\vert_V = 0\).
\end{lemma}

\begin{proof}
The pair \((\nu_+,\nu_-)\) is also an eigenvector pair for \(f^n\), so without loss of generality we can assume that \(n = 1\).
Let \(i : \widetilde{V} \to X\) be the composition of the normalization of \(V\) with its inclusion into \(X\).  Then
\((\widetilde{f})^\ast (i^\ast \nu_+) = i^\ast f^\ast \nu_+ = i^\ast (\lambda_1(f) \nu_+) = \lambda_1(f) (i^\ast \nu_+)\), so that \(i^\ast \nu_+\) is a \(\lambda_1(f)\)-eigenvector for \(\widetilde{f}^\ast\).  Since the spectral radius of \(\widetilde{f}^\ast\) is \(\lambda_1(f \vert_V) < \lambda_1(f)\), this is impossible unless \(i^\ast \nu_+ = 0\), which means that \(\nu_+ \vert_V = 0\). 
\end{proof}

\begin{definition}
Suppose that \(f : X \to X\) is an automorphism of a normal projective variety.  Then \(E(f)\) is the subset of \(X\) defined by
\[
E(f) = \bigcup \set{ V : \text{$\dim V \geq 1$, $V$ is $f$-periodic, $\lambda_1(f \vert_V) < \lambda_1(f)$, and $\lambda_1(f^{-1} \vert_V) < \lambda_1(f^{-1})$}}.
\]
\end{definition}

\begin{theorem}
\label{dbigperiodic}
Suppose that \(X\) is a normal projective variety and that \(f : X \to X\) is an automorphism satisfying Condition (B).  Then \(E(f)\) is not Zariski dense in \(X\).
\end{theorem}

\begin{proof}
Let \((\nu_+,\nu_-)\) be an eigenvector pair for \(f\) with
\(\nu = \nu_+ + \nu_-\) big, and suppose that \(V\) is irreducible and
\(f\)-periodic.  By Lemma~\ref{lem:autorestrict}, we have
\(\nu_+\vert_V = \nu_-\vert_V = 0\) in \(N^1(V)\), which implies that
\(V \subset \BB_+(\nu)\). Thus \(E(f) \subseteq \BB_+(\nu)\). Since \(\nu\) is big, \(\BB_+(\nu)\) is a proper Zariski-closed subset of \(X\), and the claim follows.
\end{proof}

\begin{remark}
If \(\lambda_1(f\vert_V) = 1\), then \(\lambda_{\dim V -1}(f\vert_V) = 1\) by log concavity of dynamical degrees, so \(\lambda_1(f^{-1}\vert_V) = 1\) by Theorem \ref{dyndegstuff} (\ref{dyndegstuff::1}). As a result, \(\BB_+(\nu)\) contains all subvarieties of \(X\) for which the dynamical degree \(\lambda_1(f\vert_V)\) drops to \(1\).
\end{remark}

\begin{example}
Let \(g : S \to S\) be an automorphism of a K3 surface satisfying \(\lambda_1(g) > 1\), and let \(f = g \times \id: S \times \PP^1 \to S \times \PP^1\), which satisfies \(\lambda_1(f) = \lambda_1(g)\).  If \(p\) is any periodic point of \(g\), then \(V = p \times \PP^1\) is \(f\)-periodic, and satisfies \(\lambda(f \vert_V) = 1\), so that \(V \subset E(f)\).  Since the \(g\)-periodic points are dense on \(S\), the set \(E(f)\) is Zariski dense.  However, \(f\) does not satisfy Condition (B), so Theorem~\ref{dbigperiodic} is not applicable.
\end{example}

\begin{remark}
When \(f\) satisfies Condition (B), we do not know in general whether \(E(f) = \BB_+(\nu)\), or whether \(E(f)\) is always Zariski closed.
\end{remark}

\begin{proposition}
\label{prop:descend-lambda-geq1}
Suppose that \(X\) has klt singularities (e.g.\ that \(X\) is smooth), that \(K_X \equiv 0\), and that \(f : X \to X\) is an automorphism of satisfying \(\lambda_1(f) > 1\) and Condition (B).  Then there exists a birational morphism \(\pi : X \to Y\) such that \(f\) descends to an automorphism \(g : Y \to Y\), and \(\pi\) contracts every connected component of \(E(f)\) to a point.
\end{proposition}

\begin{proof}
Since \(f\) satisfies Condition (B), there is an eigenvector pair \((\nu_+,\nu_-)\) with \(\nu = \nu_+ + \nu_-\) big.  When \(\nu\) is represented by a \(\QQ\)-divisor \(D\), the claim follows quickly from Kawamata's basepoint-free theorem: \(D\) is semi-ample, and we take \(\pi\) to be the corresponding contraction.  However, since \(\nu\) does not typically have a \(\QQ\)-divisor representative, we must resort to other methods, and we realize \(Y\) as the log canonical model of a klt pair \((X,\Delta)\) with \(\Delta \equiv \epsilon \nu\).  

Since \(\nu\) is big, we may find \(\epsilon > 0\) and an effective \(\RR\)-divisor \(\Delta \equiv \epsilon \nu\) such that \((X,\Delta)\) is klt~\cite[Corollary 2.35]{kollarmori}.  Note that \(K_X+ \Delta = \Delta\) is nef.  It follows from \cite{bchm} that there exists a log canonical model \(\pi : X \rat Y\) for the pair \((X,\Delta)\), which means that
\begin{enumerate}
\item \(\pi\) is a birational contraction (i.e.\ \(\pi\) is birational and \(\pi^{-1}\) does not contract any divisors);
\item \(\pi\) is \((K_X+\Delta)\)-negative (in the sense of ~\cite{bchm});
\item taking \(\Gamma = \pi_\ast \Delta\), we have \(K_Y+\Gamma\) ample.
\end{enumerate}

We argue now that if \(K_X+\Delta\) is big and nef, the map \(\pi\) is in fact a morphism (a standard fact, for which we do not know a convenient reference).  Take a resolution of the rational map \(\pi\):
\[
\xymatrix{
& W \ar[dl]_p \ar[dr]^q \\
X \ar@{-->}[rr]_\pi && Y
}
\]
Since \(\pi\) is \((K_X+\Delta)\)-negative, we have \(p^\ast (K_X+\Delta) = q^\ast(K_Y + \Gamma) + E\), with \( E \geq 0\).  It follows from~\cite{nakayama} that
\[
E = N_\sigma(q^\ast(K_Y + \Gamma) +E) = N_\sigma(p^\ast(K_X + \Delta)) = 0,
\]
and so \(p^\ast (K_X+\Delta) = q^\ast (K_Y + \Gamma)\).  It then follows from ~\cite[3.6.6(2)]{bchm} that \(\pi\) is a morphism, and that \(K_X+\Delta = \pi^\ast A\), where \(A\) is ample.  Since \(K_X \equiv 0\) by assumption, this means that \(\epsilon \nu \equiv \pi^\ast A\).

Suppose that \(V\) is an irreducible component of \(E(f)\).  By Lemma~\ref{lem:autorestrict}, we have \(\nu_+\vert_V = \nu_-\vert_V = 0\), and so \(\nu\vert_V = 0\).  Since \(D = \pi^\ast A\), it follows that all such subvarieties \(V\) are contracted to points by \(\pi\).

It remains to check that \(f\) induces an automorphism \(g : Y \to Y\).  We claim first that every subvariety contracted to a point by \(\pi\) is also contracted by \(\pi \circ f\).  The varieties contracted by \(\pi\) are precisely those \(V\) for which \(\nu \vert_V = 0\).  Since \(\nu_+\) and \(\nu_-\) are nef, this is possible only if \(\nu_+\vert_V = \nu_-\vert_V = 0\).  The varieties contracted by \(f \circ \pi\) are those on which \((f^{-1})^\ast(\nu) = \lambda_1(f)^{-1} \nu_+ + \lambda_1(f^{-1})^{-1} \nu_-\) restricts to \(0\), which is the same set of varieties.    

The map \(\pi : X \to Y\) is birational with \(Y\) normal and so satisfies \(\pi_\ast \O_X = \O_Y\) by Zariski's main theorem, and since \(f \circ \pi\) contracts every fiber of \(\pi\),
it follows from the rigidity lemma~\cite[Lemma 1.15(b)]{debarre} that it factors through \(\pi\).  This yields a map \(g : S \to S\) with \(f \circ \pi = \pi \circ g\).  An inverse to \(g\) is obtained by applying the same argument to \(f^{-1}\).
\end{proof}

Theorem~\ref{intro:contract} then follows from Theorem~\ref{dbigperiodic} and Proposition~\ref{prop:descend-lambda-geq1}, since automorphisms of hyper-K\"ahler varieties satisfy Condition (B) by Lemma \ref{aandbhold}.

\begin{example}[{\cite[Example 5.2]{lobianco}}]
Suppose that \(f : S \to S\) is an automorphism of a K3 surface with \(\lambda_1(f) > 1\).  Let \(X  = \Hilb^n(S)\) be the corresponding Hilbert scheme of \(n\) points on \(S\).  There is an induced automorphism \(f^{[n]}  : X \to X\), and \(\lambda_1(f^{[n]}) = \lambda_1(f)\).

The \(f\)-periodic points \(p\) on \(S\) are Zariski dense \cite{cantat-k3}, giving rise to \(f\)-periodic subvarieties \(V\) on \(X\) of any even codimension, as the images of \(p \times \cdots \times p \times S \times \cdots \times S\) in \(X\).  These \(f^{[n]}\)-periodic subvarieties are Zariski dense, but they satisfy \(\lambda_1(f\vert_V) = \lambda_1(f)\) and \(\lambda_1(f^{-1}\vert_V) = \lambda_1(f^{-1})\) and so do not contradict Theorem \ref{dbigperiodic}.

If \(f : S \to S\) has an invariant curve \(C\), then the image of \(C \times \cdots \times C\) in \(\Hilb^n(S)\) is an \(n\)-dimensional subvariety \(V\) of \(\Hilb^n(S)\) on which \(\lambda_1(f^{[n]} \vert_V) = 1\), and so the set \(E(f)\) is not always empty.
\end{example}

As a consequence of Theorem \ref{thm:hk}, we reduce Conjecture \ref{conj:main} for automorphisms of smooth varieties $X$ with $K_X\equiv0$ to the case of Calabi--Yau varieties. This is done in Corollary \ref{cor:min Kod 0 <-> CYn}.

\begin{proof}[Proof of Corollary \ref{cor:min Kod 0 <-> CYn}]
Let $X$ be a smooth projective $\QQbar$-variety with numerically trivial canonical class, and $f\colon X\to X$ an automorphism. 
By \cite[Proposition 3.1]{BGRS}, there is an abelian variety $A$, Calabi--Yau varieties $Y_i$, and hyper-K\"ahler varieties $Z_j$ all defined over $\overline{\QQ}$, and there is a finite \'etale cover $\pi\colon\widetilde{X}\to X$ where $\widetilde{X}=A\times\prod_i Y_i\times \prod_j Z_j$. Applying condition (3) of \cite[Proposition 3.1]{BGRS} to $f\circ\pi$ yields a map $\widetilde{f}$ making the diagram
\[
\xymatrix{
\widetilde{X}\ar[r]^-{\widetilde{f}}\ar[d]^-{\pi} & \widetilde{X}\ar[d]^-{\pi}\\
X\ar[r]^-{f} & X
}
\]
commute. Since $\pi$ is finite \'etale, by degree considerations, we see $\widetilde{f}$ is an automorphism. By \cite[Lemma 3.2]{mss}, the conjecture for $f$ follows from that of $\widetilde{f}$, so we may assume $X$ itself is a product $A\times \prod_i Y_i\times \prod_j Z_j$ as above.

Recall that Conjecture \ref{conj:main} holds for $f$ if and only if it holds for an iterate of $f$. Since the $Y_i$ and $Z_j$ are simply connected, their first Betti numbers are trivial, so after possibly replacing $f$ by an iterate, we may assume by Theorem 4.6 and 
Lemma 5.1 of \cite{sano-prod-endo} that $f=f_0\times \prod_i g_i \times \prod_j h_j$ with $f_0$ an endomorphism of $A$, $g_i$ an endomorphism of $Y_i$, and $h_j$ an endomorphism of $Z_j$. Applying the same argument to $f^{-1}$, we may assume $f^{-1}=f'_0\times \prod_i g'_i \times \prod_j h'_j$. Since $\id=ff^{-1}=f_0f_0'\times\prod_i g_ig'_i\times\prod_j h_jh'_j$, it follows that $f_0^{-1}=f'_0$, $g_i^{-1}=g'_i$, and $h_j^{-1}=h'_j$; so, $f_0$, $g_i$, and $h_j$ are all automorphisms.

By \cite[Lemma 3.2]{sano-prod-endo}, the conjecture for $f$ then follows from the conjecture for $f_0$, $g_i$, and $h_j$. Conjecture \ref{conj:main} is known for abelian varieties by \cite{silverman-ab-vars}, and we proved in Theorem \ref{thm:hk} that the conjecture holds for hyper-K\"ahler varieties. Thus, Conjecture \ref{conj:main} for $f$ is reduced to that of each $g_i$, i.e.~automorphisms of Calabi--Yau varieties of dimension at most $n$.
\end{proof}

\section{Endomorphisms of Kodaira dimension 0 threefolds: Proposition \ref{prop:intro-3-fold-Kod-dim-0}}
\label{k0threefolds}

The goal of this brief section is to prove Conjecture \ref{conj:main} for all smooth 3-folds $X$ of Kodaira dimension $0$ and surjective endomorphisms $f$ with $\deg(f)>1$. 
The crux of the argument is a theorem of Fujimoto that it is possible to run the minimal model program on \(X\) while only contracting \(f\)-periodic rays.

\begin{proof}[Proof of Proposition \ref{prop:intro-3-fold-Kod-dim-0}]
By \cite[Lemma 2.3]{fujimoto}, $f$ is a finite \'etale cover and so $\chi(\O_X)=\deg(f)\chi(\O_X)$. Then $\chi(\O_X)=0$ since $\deg(f)>1$. By \cite[Corollary 4.4]{fujimoto} and its proof, we know that all extremal contractions of $X$ are of type (E1) (the inverse of the blow-up along a smooth curve), so the minimal model of $X$ is smooth, and $f$ descends to a surjective endomorphism of a minimal model of \(X\). The argument of \cite{fujimoto} is based on a run of the MMP and holds over any algebraically closed field of characteristic \(0\), so the minimal model of \(X\) is defined over \(\QQbar\). By Theorem \ref{dyndegstuff} (\ref{dyndegstuff::4}), the Kawaguchi-Silverman conjecture holds for $f$ if and only if it holds for the induced endomorphism of the minimal model of $X$. We may therefore assume $X$ itself is minimal. 

The Abundance Conjecture is known in dimension 3 by~\cite{kawamata-abundance}, so \(K_X \equiv 0\). By \cite[Proposition 3.1]{BGRS}, there is a finite \'etale cover $\pi\colon\widetilde{X}\to X$ with $\widetilde{X}=A\times\prod_i Y_i\times\prod_j Z_j$ where $A$ is an abelian variety, $Y_i$ are Calabi--Yau varieties, and $Z_j$ are hyper-K\"ahler varieties all defined over $\QQbar$. Applying condition (3) of \cite[Proposition 3.1]{BGRS} to $f\circ\pi$ yields a map $\widetilde{f}$ making the diagram
\[
\xymatrix{
\widetilde{X}\ar[r]^-{\widetilde{f}}\ar[d]^-{\pi} & \widetilde{X}\ar[d]^-{\pi}\\
X\ar[r]^-{f} & X
}
\]
commute. We see $\widetilde{f}$ is finite \'etale with $\deg(\widetilde{f})=\deg(f)$. From \cite[Main Theorem A]{fujimoto} Case 3, we know that $\widetilde{X}$ is an abelian 3-fold or $E\times Z$ with $E$ and elliptic curve and $Z$ a K3 surface; the reason $\widetilde{X}$ cannot be a Calabi--Yau 3-fold is that $\pi_1(X)$ is infinite, see \cite[Claim, pg.\ 66]{fujimoto}. By \cite[Theorem 1.3]{sano-prod-endo}, Conjecture \ref{conj:main} is known for products of abelian varieties and K3 surfaces, so it is known for $\widetilde{f}$. By \cite[Lemma 3.2]{mss}, the conjecture for $f$ follows.
\end{proof}

\section{Automorphisms of Calabi--Yau threefolds: Theorem \ref{thm:intro-CY3-aut}}
\label{sec:CY3-aut}

\begin{proof}[Proof of Theorem \ref{thm:intro-CY3-aut}]
We first handle case (\ref{intro-CY3-aut::Miyaoka}). By \cite[Lemma 7.1]{BGRS} we know that $\{D\in\Nef(X)\mid c_2(X)\cdot D\leq M\}$ is compact for all $M\geq0$. So the function $D\mapsto c_2(X)\cdot D$ achieves a minimum positive value on $N^1(X) \cap \Amp(X)$ and this value is achieved by only finitely many $D_i$. Taking the sum of these finitely many $D_i$, we obtain an ample class $A$ which is fixed by $f^*$.  It follows that some iterate \(f^n\) lies in the connected component of the identity \(\Aut^0(X) \subseteq \Aut(X)\).
Since \(X\) is a Calabi--Yau threefold, \(\dim \Aut^0(X) = \dim H^0(X,TX) = 0\), and we conclude that \(f\) has finite order, so the conjecture holds vacuously.

We now turn to case (\ref{intro-CY3-aut::semiample}). Let $\pi:X\to Y$ be the contraction map associated to $D$; since $D \cdot c_2(X) 0$, this is referred to as a $c_2$-contraction. 
Oguiso shows in \cite[Theorem 4.3]{semi-ampleness-conj} that there are only finitely many $c_2$-contractions, and so after replacing $f$ by a further iterate, we can assume $f^*[D] = [D]$. By \cite[Proposition 6.1(a)]{BGRS}, we know that $f$ descends to an automorphism $g$ of $Y$. Since $D\neq0$, we see $\dim Y>0$.

Let us first suppose that $\dim Y=1$. By hypothesis, there is a rational point $P \in X(\QQbar)$ with Zariski dense orbit under $f$, so $\pi(P)\in Y$ has Zariski dense orbit under $g$. As a result, $Y$ must be rational or an elliptic curve; since $X$ has trivial Albanese, we see $Y\simeq\PP^1$. Let $Z\subseteq \PP^1$ be the locus of points $t$ where the fiber $X_t$ is singular. Then $g(Z)=Z$. Since $Z$ is a finite set, after replacing $f$ by a further iterate, we can assume $g$ fixes $Z$ point-wise. By \cite[Theorem 0.2]{mfds-over-curves}, we know that $Z$ contains at least 3 points. It follows that $g$ is the identity since it fixes at least three points of $\PP^1$. In other words, there exists a rational function on $X$ which is invariant under some iterate of $f$, which contradicts the fact that $X$ has a point with dense orbit.

The case in which $\dim Y=2$ is an immediate consequence of Corollary \ref{cor:fibered-over-surface}; that \(Y\) is normal and \(\QQ\)-factorial is proved in~\cite[pg.\ 18]{c2-fibration-types}.

Finally, we handle the case where $\dim Y=3$, i.e., $D$ is big. Since contractions have connected fibers, $\pi$ is birational.  Then \(D = \pi^\ast H\) for some ample divisor $H$ on \(Y\).  Then \(\pi^\ast (g^\ast H) = f^\ast (\pi^\ast H) = f^\ast D = D = \pi^\ast H\), which shows that \(g^\ast H = H\), and so \(\lambda_1(g) = 1\). Theorem \ref{dyndegstuff} (\ref{dyndegstuff::4}) shows that \(\lambda_1(f) = \lambda_1(g) = 1\), and the conjecture holds for \(f\) by Remark \ref{rmk:lambda1 is 1}.
\end{proof}

\section{Mori fiber spaces}
\label{sec:auts-kappa-minus-infty}

\subsection{Automorphisms of threefold Mori fiber spaces: Theorem \ref{thm:intro-MFS-partial} (\ref{intro-MFS-partial::auts3fold})}

We prove Theorem \ref{thm:intro-MFS-partial} (\ref{intro-MFS-partial::auts3fold}) after a preliminary lemma.

\begin{definition}
A \emph{Mori fiber space} is a projective morphism \(\pi : X \to S\) such that \(X\) is terminal and \(\QQ\)-factorial, \(-K_X\) is \(\pi\)-ample, and \(\rho(X/S) = 1\).
\end{definition}

\begin{lemma}
\label{l:MFS-descend-endo}
Let $\pi\colon X\to S$ be a Mori fiber space. If $f$ is a surjective endomorphism of $X$, then after replacing \(f\) by a suitable iterate \(f^m\), we may assume that there is an endomorphism \(g : S \to S\) such that \(g \circ \pi = \pi \circ f\). If $f$ is an automorphism then $g$ is also an automorphism.
\end{lemma}
\begin{proof}
We claim first that some iterate of \(f\) maps fibers to fibers.  This is a consequence of an observation of Wi\'sniewski~\cite[Theorem 2.2]{wisniewski} (see also ~\cite[Exercise III.1.19]{kollarrationalcurves}): on a given variety, there are only finitely many \(K_X\)-negative extremal rays on the closed cone of curves \(\NEb(X)\) yielding Mori fiber space structures.

The existence of the map \(g\) is a consequence of the rigidity lemma~\cite[Lemma 1.15(b)]{debarre}, as in the proof of Proposition~\ref{prop:descend-lambda-geq1}, since a Mori fiber space necessarily satisfies \(\pi_\ast \O_X = \O_S\).
\end{proof}

\begin{theorem}
\label{mfsreduction}
Suppose that \(\pi : X \to S\) is a Mori fiber space.  Suppose that \(f : X \to X\) and \(g : Y \to Y\) are automorphisms with \(\pi \circ f = g \circ \pi\).  If Conjecture~\ref{conj:main} holds for \(g\), then it holds for \(f\).
\end{theorem}

\begin{proof}
Recall that the first relative dynamical degree is defined by
\[
\lambda_1(\pi \vert_f) = \lim_{n \to \infty} \left( (f^n)^\ast H \cdot \pi^\ast (H'^{\dim Y}) \cdot H^{\dim X - \dim Y - 1} \right)^{1/n}.
\]
Here \(\pi^\ast (H'^{\dim Y})\) is the class of a fiber of \(\pi\), and \(\pi^\ast (H'^{\dim Y}) \cdot H^{\dim X - \dim Y - 1}\) is the class of some curve in the fiber.  Since \(\pi\) is a Mori fiber space, all curves contained in fibers are proportional in \(N_1(X)\), and since \(f\) is an automorphism defined over \(\pi\), this class must be invariant under \(f\).  It follows that \(\lambda_1(\pi \vert_f) = 1\).  The claim is then a consequence of Theorem~\ref{thm:fibered-over-dim-one-less}.
\end{proof}

\begin{proof}[Proof of Theorem \ref{thm:intro-MFS-partial} (\ref{intro-MFS-partial::auts3fold})]
Let $X$ be a threefold, $f$ an automorphism of $X$, and \(\pi : X \to S\) a Mori fiber space structure. After replacing \(f\) by an iterate, by Lemma~\ref{l:MFS-descend-endo} we may assume that there is an automorphism \(g : S \to S\) such that \(\pi \circ f = g \circ \pi\).
Since \(\dim S \leq 2\) and \(g\) is an automorphism, Conjecture~\ref{conj:main} is known for \((S,g)\), and the conjecture for \((X,f)\) follows from Theorem~\ref{mfsreduction}.
\end{proof}

\subsection{Endomorphisms of rational normal scrolls: Theorem \ref{thm:intro-MFS-partial} (\ref{intro-MFS-partial::ratnormalscrolls})}
\label{subsec:general results about P2 bundles on curves}
Let $C$ be a smooth projective curve over $\QQbar$, $\E$ a vector bundle on $C$ of rank $n$, and $X=\PP_C(\E)$. By \cite[Theorem 9.6]{3264}, the Chow group of $X$ is given by
\[
\begin{split}
A^*(X) &=A^*(C)[D]/(D^n+c_1(\E)D^{n-1}+c_2(\E)D^{n-2}+\dots +c_{n}(\E))\\
 &=A^*(C)[D]/(D^n+c_1(\E)D^{n-1} F),
\end{split}
\]
where $F$ is the class of a fiber. So $A^*(X)$ is generated by the divisor classes $F$ and $D$ and we have the relations $F^2=0$, $FD^{n-1}=1$, and $D^n=-c_1(\E)D^{n-1}F=-c_1(\E)$; the second relation holds because $DF=D\vert_F$ is the class of a hyperplane on $F=\PP^{n-1}$ and so $FD^{n-1}=(D|_F)^{n-1}=1$. 

The nef cone of $X$ is given by the following, 
which generalizes a result of Miyaoka \cite[Theorem 3.1]{Miyaoka}. Recall that the slope $\mu(\E)$ is defined to be $c_1(\E)/\rank(\E)$. We let $\mu_{\min}(\E)$ and $\mu_{\max}(\E)$ denote the minimum, resp.~maximum, slope of the graded pieces appearing in the Harder-Narasimhan filtration of $\E$.

\begin{lemma}
\label{l:nef cone of PE}
$\Nef(X)$ is the cone generated by $F$ and $D-\mu_{\min}(\E)F$.
\end{lemma}
\begin{proof}
See e.g.~\cite[Lemma 4.4.1]{nakayama} or \cite[Lemma 2.1]{Fulger-proj-bundles-curves}.
\end{proof}

Given a surjective endomorphism $f$ of $X=\PP_C(\E)$, in order to verify Conjecture \ref{conj:main} we may replace $f$ by an iterate. Since the structure map $\pi\colon X\to C$ is a Mori fiber space, by Lemma \ref{l:MFS-descend-endo} we can replace \(f\) by an iterate and assume that there is an endomorphism $g$ of $C$ such that $\pi\circ f=g\circ \pi$. We assume we are in this situation throughout this section. Let
\[
\delta:=\frac{\deg(f)}{\deg(g)}.
\]

\begin{lemma}
\label{l:eigenvals of f}
The action of $f^*$ on $N^1(X)$ is given by
\[
f^*(F)=\deg(g)\,F,\quad\quad f^*(D)=(\deg(g)-\delta^{1/(n-1)})\mu_{\min}(\E)\,F+\delta^{1/(n-1)}\,D
\]
and has eigenvalues $\lambda_1(g)=\deg(g)$ and $\delta^{1/(n-1)}$. Moreover,
\[
\lambda_1(f)=\max(\lambda_1(g),\delta^{1/(n-1)}).
\]
\end{lemma}
\begin{proof}
It is clear that $F$ is an eigenvector with eigenvalue $\deg(g)=\lambda_1(g)$: since $F$ is a fiber it is of the form $\pi^{-1}(P_0)$ for a point $P_0\in C$ and we have
\[
f^*F=f^*\pi^*P_0=\pi^*g^*P_0=\pi^*(\deg(g)P_0)=\deg(g)F.
\]
Next, let $f^*D=cF+dD$. Notice that with respect to the basis $F,D$ for $N^1(X)$, the matrix for $f^*$ is upper triangular with diagonal entries $\deg(g)$ and $d$. So, the eigenvalues for $(f^p)^*$ are given by $\deg(g)^p$ and $d^p$. Since $\lambda_1(f)=\lim_{p\to\infty}\operatorname{SpecRad}((f^p)^*)^{1/p}$, we see $\lambda_1(f)=\operatorname{SpecRad}(f^*)=\max(\deg(g),d)$. So, we need only show $d=\delta^{1/(n-1)}$, i.e.~that $\deg(f)=d^{n-1}\deg(g)$. Notice that
\[
\begin{split}
\deg(f) &=\deg(f)D^{n-1}F=f_*f^*(FD^{n-1})=f_*(f^*F\cdot (f^*D)^{n-1})\\
&=\deg(g)f_*(F\cdot(cF+dD)^{n-1})=\deg(g)f_*(d^{n-1}FD^{n-1})=d^{n-1}\deg(g).
\end{split}
\]
So, we have now shown that the eigenvalues of $f^*$ are $\lambda_1(g)=\deg(g)$ and $\delta^{1/(n-1)}$, and that $\lambda_1(f)=\max(\lambda_1(g),\delta^{1/(n-1)})$.

Lastly, we must calculate $c$. To do so, we use Lemma \ref{l:nef cone of PE}. Notice that the determinant of the action of $f^*$ on $N^1(X)$ is $\deg(f)>0$ so the action is orientation-preserving. Since $f$ is finite, for all $D'$ we know $D'$ is ample if and only if $f^*D'$ is ample. As a result, the boundary rays of $\Nef(X)$ are each sent to themselves. Thus, the eigenvectors for $f^*$ are given by $F$ and $D-\mu_{\min}(\E)F$. In particular, $d(D-\mu_{\min}(\E)F)=f^*(D-\mu_{\min}(\E)F)=cF+dD-\deg(g)\mu_{\min}(\E)F$, and so
\[
c=(\deg(g)-d)\mu_{\min}(\E),
\]
proving the lemma.
\end{proof}

\begin{proposition}
\label{prop:dichotomy dynamical deg equals base or controlled slope}
One of the following holds: $\lambda_1(f)=\lambda_1(g)$ or $\mu_{\min}(\E)=-\mu(\E)$.
\end{proposition}
\begin{proof}
From Lemma \ref{l:eigenvals of f}, we know $f^*D=cF+dD$ where $c=(\deg(g)-d)\mu_{\min}(\E)$ and $d=\delta^{1/(n-1)}$. Recalling that $D^n=-c_1(\E)$, we have
\[
-\deg(f)c_1(\E) = f_*f^*(D^n)=f_*(cF+dD)^n=ncd^{n-1}-d^nc_1(\E).
\]
Substituting for $c$, we have
\[
-\deg(f)c_1(\E) = n(\deg(g)-d)\mu_{\min}(\E)d^{n-1}-d^nc_1(\E)=n(\deg(f)-d^n)\mu_{\min}(\E)-d^nc_1(\E)
\]
and so
\[
d^n(c_1(\E)+n\mu_{\min}(\E))=\deg(f)(c_1(\E)+n\mu_{\min}(\E)).
\]
Thus, $\mu_{\min}(\E)=-c_1(\E)/n=:-\mu(\E)$ or $d^n=\deg(f)$. This latter equality is equivalent to $d=\deg(g)=\lambda_1(g)$, which by Lemma \ref{l:eigenvals of f}, implies $\lambda_1(f)=\lambda_1(g)$.
\end{proof}

We next need the following basic result concerning the Harder-Narasimhan filtration.

\begin{lemma}
\label{l:min and max slopes in HN filtration}
If $\E$ is a vector bundle which is not semistable, then $\mu_{\max}(\E)>\mu(\E)>\mu_{\min}(\E)$.
\end{lemma}
\begin{proof}
Let
\[
0=\E_0\subsetneq \E_1\subsetneq \dots\subsetneq \E_{\ell-1}\subsetneq \E_\ell = \E
\]
be the Harder-Narasimhan filtration of $\E$, so that $\mu_{\max}(\E)=\mu(\E_1)$ and $\mu_{\min}(\E)=\mu(\E/\E_{\ell-1})$. By construction, $\E_1$ is the maximal destabilizing subbundle of $\E$, i.e.~for all subbundles $0\neq\F\subseteq\E$ we have: (i) $\mu(\E_1)\geq\mu(\F)$ and (ii) if $\mu(\E_1)=\mu(\F)$, then $\F\subseteq\E_1$. So, we see $\mu(\E_1)\geq\mu(\E)$ and we cannot have equality since then we would have $\E=\E_1$ which is not possible as $\E_1$ is semistable and $\E$ is not. We have therefore shown $\mu_{\max}(\E)>\mu(\E)$.

To show $\mu(\E)>\mu_{\min}(\E)$, we induct on $\ell$. We first recall the general result which follows immediately from the definition of slope: if
\[
0\to\F'\to\F\to\F''\to0
\]
is a short exact sequence of non-trivial vector bundles, then $\mu(\F')>\mu(\F)$ if and only if $\mu(\F)>\mu(\F'')$.

Since $\E$ is not semistable, we have $\ell\geq2$. When $\ell=2$ we have a short exact sequence
\[
0\to\E_1\to\E\to\E_2/\E_1\to0
\]
and since we have already shown $\mu(\E_1)>\mu(\E)$, we know $\mu(\E)>\mu(\E_2/\E_1)=\mu_{\min}(\E)$.

Next suppose $\ell\geq3$. Then
\[
0\neq\E_2/\E_1\subsetneq \dots\subsetneq \E_{\ell-1}/\E_1\subsetneq \E/\E_1
\]
is the Harder-Narasimhan filtration of $\E/\E_1$; it has length $\ell-1\geq2$ and so $\E/\E_1$ not semistable. Then by induction, $\mu(\E/\E_1)>\mu_{\min}(\E/\E_1)=\mu(\E/\E_{\ell-1})=\mu_{\min}(\E)$. Since we have shown $\mu(\E_1)>\mu(\E)$, we know $\mu(\E)>\mu(\E/\E_1)$ and so $\mu(\E)>\mu_{\min}(\E)$.
\end{proof}

\begin{corollary}
\label{cor:reduction to case of semistable deg 0}
Let $C$ be a smooth curve. Then the following are equivalent:
\begin{enumerate}
\item Conjecture \ref{conj:main} holds for all surjective endomorphisms of varieties of the form $\PP_C(\E)$
\item Conjecture \ref{conj:main} holds for all surjective endomorphisms of varieties of the form $\PP_C(\E)$ with $\E$ semistable of degree 0.
\end{enumerate}
\end{corollary}
\begin{proof}
By Proposition \ref{prop:dichotomy dynamical deg equals base or controlled slope}, we know $\lambda_1(f)=\lambda_1(g)$ or $\mu_{\min}(\E)=-\mu(\E)$. Suppose $\lambda_1(f)=\lambda_1(g)$ and $P\in X(\QQbar)$ has dense orbit under $f$. Then $\pi(P)$ has dense orbit under $g$, so $\alpha_g(\pi(P))=\lambda_1(g)$ since the conjecture is known for curves. Then Lemma \ref{l:semi-ample-ht} shows $\alpha_f(P)\geq\alpha_g(\pi(P))=\lambda_1(g)=\lambda_1(f)$, and hence $\alpha_f(P)=\lambda_1(f)$ by Remark \ref{rmk:lambda1 is 1}.

We next turn to the case where $\mu_{\min}(\E)=-\mu(\E)$. Since $X=\PP(\E\otimes\L)$ for any line bundle $\L$, choosing $\L$ with sufficiently negative degree, we can assume $\mu(\E)<0$. If $\E$ is not semistable, then by Lemma \ref{l:min and max slopes in HN filtration} we have $\mu(\E)>\mu_{\min}(\E)=-\mu(\E)$ which is a contradiction. So, $\E$ must be semistable, in which case $\mu(\E)=\mu_{\min}(\E)=-\mu(\E)$, so $\mu(\E)=0$, i.e.~$\deg\E=0$.
\end{proof}

We are now ready to prove Conjecture \ref{conj:main} in the case where $C=\PP^1$, i.e.~the case of rational normal scrolls.

\begin{proof}[Proof of Theorem \ref{thm:intro-MFS-partial} (\ref{intro-MFS-partial::ratnormalscrolls})]
By Corollary \ref{cor:reduction to case of semistable deg 0}, we need only prove the conjecture for semistable degree 0 vector bundles on $\PP^1$. Such vector bundles are all trivial, so $X=\PP^1\times\PP^{n-1}$ in which case the conjecture holds by \cite[Theorem 1.3]{sano-prod-endo}.
\end{proof}

\singlespacing

\begin{thebibliography}{10}

\bibitem{amerikverbitsky}
Ekaterina Amerik and Misha Verbitsky, \emph{Construction of automorphisms of
  hyperk\"ahler manifolds}, Compos. Math. \textbf{153} (2017), no.~8,
  1610--1621.

\bibitem{beauville}
Arnaud Beauville, \emph{Vari\'et\'es {K}\"ahleriennes dont la premi\`ere classe
  de {C}hern est nulle}, J. Differential Geom. \textbf{18} (1983), no.~4,
  755--782 (1984).

\bibitem{BGRS}
J.~Bell, D.~Ghioca, Z.~Reichstein, and M.~Satriano, \emph{On the
  {M}edvedev-{S}canlon conjecture for minimal threefolds of non-negative
  {K}odaira dimension}, New York Journal of Mathematics (2017).

\bibitem{Einstein-mfds}
Arthur~L. Besse, \emph{Einstein manifolds}, Classics in Mathematics,
  Springer-Verlag, Berlin, 2008, Reprint of the 1987 edition. \MR{2371700}

\bibitem{lobianco}
Federico~Lo Bianco, \emph{On the primitivity of birational transformations of
  irreducible holomorphic symplectic manifolds}, Int. Math. Res. Not. IMRN
  (2017).

\bibitem{bchm}
Caucher Birkar, Paolo Cascini, Christopher~D. Hacon, and James McKernan,
  \emph{Existence of minimal models for varieties of log general type}, J.
  Amer. Math. Soc. \textbf{23} (2010), no.~2, 405--468.

\bibitem{birkhoff}
Garrett Birkhoff, \emph{Linear transformations with invariant cones}, Amer.
  Math. Monthly \textbf{74} (1967), 274--276.

\bibitem{cantat-k3}
Serge Cantat, \emph{Dynamique des automorphismes des surfaces {$K3$}}, Acta
  Math. \textbf{187} (2001), no.~1, 1--57.

\bibitem{cantatsurvey}
\bysame, \emph{Dynamics of automorphisms of compact complex surfaces},
  Frontiers in complex dynamics, Princeton Math. Ser., vol.~51, Princeton Univ.
  Press, Princeton, NJ, 2014, pp.~463--514.

\bibitem{rat-maps-normal-vars}
Nguyen-Bac Dang, \emph{Degrees of iterates of rational maps on normal
  projective varieties},  (2017).

\bibitem{debarre}
Olivier Debarre, \emph{Higher-dimensional algebraic geometry}, Universitext,
  Springer-Verlag, New York, 2001.

\bibitem{dinhnguyen}
Tien-Cuong Dinh and Vi\^et-Anh Nguy\^en, \emph{Comparison of dynamical degrees
  for semi-conjugate meromorphic maps}, Comment. Math. Helv. \textbf{86}
  (2011), no.~4, 817--840.

\bibitem{dinhnguyentruong}
Tien-Cuong Dinh, Vi\^et-Anh Nguy\^en, and Tuyen~Trung Truong, \emph{On the
  dynamical degrees of meromorphic maps preserving a fibration}, Commun.
  Contemp. Math. \textbf{14} (2012), no.~6, 1250042, 18.

\bibitem{dinhsibonygreen}
Tien-Cuong Dinh and Nessim Sibony, \emph{Green currents for holomorphic
  automorphisms of compact {K}\"ahler manifolds}, J. Amer. Math. Soc.
  \textbf{18} (2005), no.~2, 291--312.

\bibitem{dinhsibony}
\bysame, \emph{Une borne sup\'erieure pour l'entropie topologique d'une
  application rationnelle}, Ann. of Math. (2) \textbf{161} (2005), no.~3,
  1637--1644.

\bibitem{elmnp1}
Lawrence Ein, Robert Lazarsfeld, Mircea Musta\c{t}\u{a}, Michael Nakamaye, and
  Mihnea Popa, \emph{Asymptotic invariants of base loci}, Ann. Inst. Fourier
  (Grenoble) \textbf{56} (2006), no.~6, 1701--1734.

\bibitem{3264}
David Eisenbud and Joe Harris, \emph{3264 and all that---a second course in
  algebraic geometry}, Cambridge University Press, Cambridge, 2016.
  \MR{3617981}

\bibitem{fujimoto}
Yoshio Fujimoto, \emph{Endomorphisms of smooth projective 3-folds with
  non-negative {K}odaira dimension}, Publ. Res. Inst. Math. Sci. \textbf{38}
  (2002), no.~1, 33--92.

\bibitem{Fulger-proj-bundles-curves}
Mihai Fulger, \emph{The cones of effective cycles on projective bundles over
  curves}, Math. Z. \textbf{269} (2011), no.~1-2, 449--459.

\bibitem{cybook}
M.~Gross, D.~Huybrechts, and D.~Joyce, \emph{Calabi-{Y}au manifolds and related
  geometries}, Universitext, Springer-Verlag, Berlin, 2003, Lectures from the
  Summer School held in Nordfjordeid, June 2001.

\bibitem{hindry-silverman}
Marc Hindry and Joseph~H. Silverman, \emph{Diophantine geometry}, Graduate
  Texts in Mathematics, vol. 201, Springer-Verlag, New York, 2000, An
  introduction.

\bibitem{hindrysilverman}
\bysame, \emph{Diophantine geometry}, Graduate Texts in Mathematics, vol. 201,
  Springer-Verlag, New York, 2000, An introduction.

\bibitem{surface-aut}
Shu Kawaguchi, \emph{Projective surface automorphisms of positive topological
  entropy from an arithmetic viewpoint}, Amer. J. Math. \textbf{130} (2008),
  no.~1, 159--186.

\bibitem{dynamical-equals-arithmetic}
Shu Kawaguchi and Joseph~H. Silverman, \emph{Examples of dynamical degree
  equals arithmetic degree}, Michigan Math. J. \textbf{63} (2014), no.~1,
  41--63.

\bibitem{ks-jordan}
\bysame, \emph{Dynamical canonical heights for {J}ordan blocks, arithmetic
  degrees of orbits, and nef canonical heights on abelian varieties}, Trans.
  Amer. Math. Soc. \textbf{368} (2016), no.~7, 5009--5035.

\bibitem{dynamical-arithmetic-rat-maps}
\bysame, \emph{On the dynamical and arithmetic degrees of rational self-maps of
  algebraic varieties}, J. Reine Angew. Math. \textbf{713} (2016), 21--48.

\bibitem{kawamataabelian}
Yujiro Kawamata, \emph{Characterization of abelian varieties}, Compositio Math.
  \textbf{43} (1981), no.~2, 253--276.

\bibitem{kawamata-abundance}
\bysame, \emph{Abundance theorem for minimal threefolds}, Invent. Math.
  \textbf{108} (1992), no.~2, 229--246.

\bibitem{kollarrationalcurves}
J{\'a}nos Koll{\'a}r, \emph{Rational curves on algebraic varieties}, Ergebnisse
  der Mathematik und ihrer Grenzgebiete. 3. Folge. A Series of Modern Surveys
  in Mathematics [Results in Mathematics and Related Areas. 3rd Series. A
  Series of Modern Surveys in Mathematics], vol.~32, Springer-Verlag, Berlin,
  1996.

\bibitem{kollarmori}
J\'anos Koll\'ar and Shigefumi Mori, \emph{Birational geometry of algebraic
  varieties}, Cambridge Tracts in Mathematics, vol. 134, Cambridge University
  Press, Cambridge, 1998, With the collaboration of C. H. Clemens and A. Corti,
  Translated from the 1998 Japanese original.

\bibitem{lazarsfeld}
Robert Lazarsfeld, \emph{Positivity in algebraic geometry. {I}},
  Springer-Verlag, Berlin, 2004.

\bibitem{CY-semiample}
V.~Lazic, K.~Oguiso, and T.~Peternell, \emph{Nef line bundles on {C}alabi-{Y}au
  thereefolds, i},  (2016).

\bibitem{matsuki-mmp}
Kenji Matsuki, \emph{Introduction to the {M}ori program}, Universitext,
  Springer-Verlag, New York, 2002.

\bibitem{on-upper-bounds-arithmetic-degrees}
Yohsuke Matsuzawa, \emph{On upper bounds of arithmetic degrees},  (2017).

\bibitem{mss}
Yohsuke Matsuzawa, Kaoru Sano, and Takahiro Shibata, \emph{Arithmetic degrees
  and dynamical degrees of endomorphisms on surfaces},  (2017).

\bibitem{Miyaoka}
Yoichi Miyaoka, \emph{The {C}hern classes and {K}odaira dimension of a minimal
  variety}, Algebraic geometry, {S}endai, 1985, Adv. Stud. Pure Math., vol.~10,
  North-Holland, Amsterdam, 1987, pp.~449--476. \MR{946247}

\bibitem{nakayama}
Noboru Nakayama, \emph{Zariski-decomposition and abundance}, MSJ Memoirs,
  vol.~14, Mathematical Society of Japan, Tokyo, 2004.

\bibitem{nakayamazhang}
Noboru Nakayama and De-Qi Zhang, \emph{Building blocks of \'etale endomorphisms
  of complex projective manifolds}, Proc. Lond. Math. Soc. (3) \textbf{99}
  (2009), no.~3, 725--756.

\bibitem{c2-fibration-types}
Keiji Oguiso, \emph{On algebraic fiber space structures on a {C}alabi-{Y}au
  {$3$}-fold}, Internat. J. Math. \textbf{4} (1993), no.~3, 439--465, With an
  appendix by Noboru Nakayama.

\bibitem{semi-ampleness-conj}
\bysame, \emph{On the finiteness of fiber-space structures on a {C}alabi-{Y}au
  3-fold}, J. Math. Sci. (New York) \textbf{106} (2001), no.~5, 3320--3335,
  Algebraic geometry, 11.

\bibitem{oguiso-hk}
\bysame, \emph{A remark on dynamical degrees of automorphisms of hyperk\"ahler
  manifolds}, Manuscripta Math. \textbf{130} (2009), no.~1, 101--111.

\bibitem{oguisorhois2}
\bysame, \emph{Automorphism groups of {C}alabi-{Y}au manifolds of {P}icard
  number 2}, J. Algebraic Geom. \textbf{23} (2014), no.~4, 775--795.

\bibitem{sano-prod-endo}
Kaoru Sano, \emph{Dynamical degree and arithmetic degree of endomorphisms on
  product varieties},  (2016).

\bibitem{shibatapic2}
Takahiro Shibata, \emph{Ample canonical heights for endomorphisms on projective
  varieties},  (2017).

\bibitem{silvermank3}
Joseph~H. Silverman, \emph{Rational points on {$K3$} surfaces: a new canonical
  height}, Invent. Math. \textbf{105} (1991), no.~2, 347--373.

\bibitem{silverman-ab-vars}
\bysame, \emph{Arithmetic and dynamical degrees on abelian varieties}, J.
  Th\'eor. Nombres Bordeaux \textbf{29} (2017), no.~1, 151--167.

\bibitem{truongreldyn}
Tuyen Truong, \emph{({R}elative) dynamical degrees of rational maps over an
  algebraic closed field}, pre-print (2015).

\bibitem{mfds-over-curves}
Eckart Viehweg and Kang Zuo, \emph{On the isotriviality of families of
  projective manifolds over curves}, J. Algebraic Geom. \textbf{10} (2001),
  no.~4, 781--799.

\bibitem{wisniewski}
Jaroslaw Wi{\'s}niewski, \emph{On contractions of extremal rays of {F}ano
  manifolds.}, Journal f\"ur die reine und angewandte Mathematik \textbf{417}
  (1991), 141--158.

\bibitem{dynamics-auts-zhang}
De-Qi Zhang, \emph{Dynamics of automorphisms on projective complex manifolds},
  J. Differential Geom. \textbf{82} (2009), no.~3, 691--722.

\bibitem{dqzhangrhois2}
\bysame, \emph{Birational automorphism groups of projective varieties of
  {P}icard number two}, Automorphisms in birational and affine geometry,
  Springer Proc. Math. Stat., vol.~79, Springer, Cham, 2014, pp.~231--238.

\end{thebibliography}

\end{document}